\documentclass[reqno, oneside]{amsart}
\usepackage{Macros} 

\usepackage{esint}

\DeclareMathOperator{\essinf}{ess inf}

\usepackage{resmes}

\title{An Aronson-Bénilan / Li-Yau estimate in the JKO scheme in small dimension}
\author{Fanch Coudreuse}
\date{January 2026}

\begin{document}

\begin{abstract}
    \noindent
    We derive an Aronson-Bénilan / Li-Yau estimate in the JKO scheme associated to the porous-medium, heat, and fast-diffusion equations, in dimensions $1$ and $2$, and on simple domains (cubes, quarter-space, half-spaces, whole space, and the torus). Our method is based on a maximum principle for the determinant of the Hessian of Brenier potentials, iterated as a one-step improvement along the scheme. As a consequence, we obtain local $L^\infty$ bounds on the density, uniform in the time step, consistent with the continuous-time result. As a byproduct, we rigorously derive the optimality conditions in the fast-diffusion case, filling a gap in the literature.
\end{abstract}

\maketitle

\section{Introduction}

    This paper focuses on the JKO scheme approximation to the porous-medium, heat and fast-diffusion equations
\begin{equation*}
    \partial_t \rho = \Delta \rho^m
\end{equation*}
on a convex domain $\Omega$, or in the torus $\bb{T}^d$, in the space of probability measures with finite second moment, with Neumann boundary condition and given initial data. The different names for this equation refer to the different regimes for the values of $m$:
\begin{itemize}
    \item The case $m > 1$ corresponds to the porous-medium equation. It exhibits slow-diffusion-type behavior, and a free-boundary can appear if one starts from a compactly supported initial datum. We refer to the monograph \cite{porous-medium} by Vazquez for the general theory of this equation.
    \item The linear case $m=1$ corresponds to the classical heat equation, where strong smoothing effects occur, and solutions are instantaneously strictly positive.
    \item Finally, the case $0 < m < 1$ corresponds to the so-called fast-diffusion regime. Solutions are again strictly positive for positive time. In this setting, one can distinguish further regimes for the parameter $m$. In our analysis, two of them will be relevant: 
    \begin{itemize}
        \item The regime $m > m_c^1 := 1 - \frac{2}{d}$ to ensure well-posedness in the space of probability measures.
        \item The regime $m > m_c^2 := 1 - \frac{2}{d+2}$ corresponding to existence of solutions in the space of probability measures with finite second moment.
    \end{itemize}
    In both cases, one can check using Barenblatt profiles (corresponding to Dirac initial datum) that this threshold is sharp. We refer to the survey \cite{Fast_Diffusion_Survey} by Bonforte and Figalli for more results on this equation.  
\end{itemize}
Note that all these equations are particular instances of the general filtration equation $\partial_t \rho = \Delta \phi(\rho)$. 

Since the seminal work of Jordan, Kinderlehrer and Otto \cite{JKO}, it is now well understood that a large class of diffusion equations posed on the space of probability measures can be interpreted as the gradient flow on the Wasserstein space. While this can be made rigorous using the general theory of gradient flows in metric spaces developed by Ambrosio, Gigli and Savar\'e in \cite{AGS}, another classical approach to tackle this interpretation is to use the JKO (Jordan-Kinderlehrer-Otto) scheme, which can be seen as an implicit Euler scheme in the Wasserstein space. This scheme takes the following form: given a time step $\tau$ and a functional $\scr{E}$ over the space of probability measures with finite second moment, one constructs a sequence $(\rho_k^\tau)_{k \geq 0}$ by fixing $\rho_0 \in \cl{P}_2(\Omega)$ and iterating
\begin{equation*}
    \rho_k^\tau \in \rm{argmin}_{\eta \in \cl{P}_{2}(\Omega)} \scr{E}[\eta] + \frac{1}{2 \tau} W_2^2(\eta,\rho_k^\tau)
\end{equation*}
Then one expects that the curve obtained by interpolation of the values of the sequence converges to a weak solution to the equation 
\begin{equation*}
    \partial_t \rho = \nabla \cdot \left ( \rho \nabla \frac{\delta \scr{E}}{\delta \rho}[\rho] \right ) 
\end{equation*}
with $\rho_0$ as initial datum, and suitable Neumann boundary conditions, where $\frac{\delta \scr{E}}{\delta \rho}$ is the first variation of $\scr{E}$ with respect to linear perturbation of the measures. This convergence can be made rigorous using the general theory of \cite{AGS}, or by \cite[Chapter 8]{OT_Filippo}, provided that the scheme exists, the initial datum is of finite energy (i.e. $\scr{E}[\rho_0] < +\infty$), and the functional $\scr{E}$ admits $\lambda$-convexity with respect to Wasserstein geodesics.  

For our setting of interest, the functional $\scr{E}$ is taken to be of the form $\scr{E}_m[\rho] = \int_\Omega f_m(\rho) \dd{\cl{L}^d}$ for $\rho \ll \cl{L}^d$, where $f_m(z) = \frac{1}{m-1} z^m$ for $m \neq 1$, and $f_1(z) = z \log z$. This functional falls into the general theory of geodesically convex functionals provided that $m \geq m_c^{\rm{geo}} = 1 - \frac{1}{d}$, in which case $f_m$ satisfies the McCann conditions \cite{McCann_conditions} \cite[Section 7.3.2]{OT_Filippo}.

A natural question in the study of the JKO scheme is whether the qualitative and quantitative properties known for the continuous flow can be recovered at the discrete level, uniformly in the time step $\tau$. Such results are desirable for at least two reasons: they provide robustness of the scheme in recovering the behavior of the continuous in time equation, and they can be leveraged to improve convergence rates. Several properties of the continuous flow have been investigated in this direction: Lipschitz and continuity estimates \cite{Lee_Lipschitz} \cite{Lip_JKO_Filippo} \cite{Caillet_Continuity}, BV and Sobolev estimates with applications to $L^2_t H^2_x$-convergence \cite{BV_estimate_JKO} \cite{LpSoboJKO} \cite{Sobo_Keller_Segel_JKO} \cite{L2H2}, or comparison principle and $L^1$-contraction \cite{L1_contraction} \cite{Comparison_Maxime} \cite{Ultra_Fast_Diffusion_JKO}.  

One particularly desirable estimate is the Aronson-B\'enilan / Li-Yau estimate. It states that, for $m > m_c^1$, and for any positive solution $\rho$ to $\partial_t \rho = \Delta \rho^m$, the pressure variable \textemdash defined by $p = \frac{m}{m-1} \rho^{m-1}$ for $m \neq 1$, and $p = \log \rho$ for $m =1$ \textemdash, satisfies the sub-harmonic lower bound: 
\begin{equation}
    \Delta p \geq -\frac{\alpha_{d,m}}{t} \qquad \alpha_{d,m} = \frac{d}{d(m-1)+ 2}
\end{equation}
In the context of the heat equation, this inequality was proved by Li and Yau \cite{Li-Yau} \cite{Yau-Harnack} and bears their name. We note that in this context, on $\bb{R}^d$ or $\bb{T}^d$, Hamilton \cite{Hamilton_Matrix_Ineq} derived the stronger semi-convexity estimate $D^2 \log \rho \succeq -\frac{1}{2 t}$, called the Li-Yau-Hamilton matrix inequality. The remaining cases were tackled by Aronson and B\'enilan \cite{Aronson_AB_estimate} \cite{AB_estimate} (see \cite[Chapter 9]{porous-medium} for a review of this estimate in the context of the porous-medium equation). Subsequently, this estimate has been extended to more general frameworks: smooth manifolds \cite{AB_Estimate_Manifold}, filtration equations \cite{AB_Estimate_Filtrat}, with extension to $L^p$-version, eventually with a source term \cite{AB_Estimate_Lp}, or to the Keller-Segel model in \cite{AB_Estimate_Keller_Segel}. This estimate is fundamental, as it is the cornerstone for deriving $L^1$-$L^\infty$-regularization effects for the equation (see Lemma \ref{lemma: L1-Linfty_regularization}), and for studying the regularity of the solution \cite{AB_estimate} \cite{Continuity_PME_Caff_Friedman} and its free-boundary \cite{Reg_Free_Boundary_PME_Caff_Friedman}. 

The first study of such an inequality in the JKO scheme was performed by P.W. Lee \cite{Lee_Harnack} for the heat equation in the torus, where he proved that a version of the Li-Yau-Hamilton inequality holds at the level of the JKO scheme, at least for regular initial data. This was then extended by the author in \cite{Li-Yau-Hamilton_JKO}, still in the torus, for a more general class of equations of granular-medium type: $\partial_t \rho = \Delta \rho + \nabla \cdot (\rho [\nabla V + \nabla W * \rho])$, under no assumptions on the initial datum; this estimate was then used to derive $L^2_{t,loc} H^2_x$-strong convergence of the scheme. In this paper, we take a first step toward a proof of an Aronson-B\'enilan estimate in the JKO scheme by focusing on small dimension and simple domains. 

\subsection{Main Result}

    We consider an iteration of the JKO scheme starting from some measure $\rho_0 \in \cl{P}_2(\Omega)$:
    \begin{equation*}
        \rho_{k+1}^\tau \in \argmin_{\eta \in \cl{P}_{2,ac}(\Omega)} \scr{E}_m[\eta] + \frac{1}{2 \tau} W_2^2(\eta,\rho_k^\tau)
    \end{equation*}
    assuming $m > m_c^1$, and additionally $m > m_c^2$ if $\Omega$ is unbounded. Although the super-linear case $m \geq 1$ is well understood and can be found in most references to the topic \cite{JKO} \cite{Otto_PME} \cite{OT_Filippo}, the case $m < 1$, to the best of our knowledge, has not been fully treated in the literature, due to the non super-linear behavior of the function $f_m$ used to define the entropy. We therefore devote some time in Section 3 to fill this gap, proving existence, uniqueness, and deriving optimality conditions in bounded domains. 
    
    Our main result, namely the Aronson-B\'enilan / Li-Yau estimate in the JKO scheme, can then be stated. Defining the discrete pressure variable $p_k^\tau := \frac{m}{m-1} (\rho_k^\tau)^{m-1}$ for $m \neq 1$, and $p_k^\tau = \log \rho_k^\tau$, it takes the following form:

    \begin{theorem}[Aronson-B\'enilan in JKO Scheme]
        Suppose that $\Omega$ is either: the torus, a cube, a quarter space, a half-space, or the whole space, in dimension $d=1$ or $2$. Then for all $k \geq 1$, $u_k^\tau := \tau p_k^\tau + \frac{1}{2} |\cdot|^2$ is convex finite on $\Omega$, and there exists a universal sequence $(X_k)_{k \geq 1}$ valued in $[0,1]$, depending only on $m,d$, such that, in the Monge-Amp\`ere sense
        \begin{equation}
            \det(D^2 u_k^\tau)^\frac{1}{d} \geq 1 - X_k
        \end{equation}
        Furthermore, as $k \to +\infty$ we have
        \begin{equation}
            X_k \sim \frac{1}{d(m-1) + 2} \cdot \frac{1}{k}
        \end{equation}
    \end{theorem}

    By "the Monge-Amp\`ere sense", we mean that the inequality should be interpreted as a lower bound on the Monge-Amp\`ere measure associated with $u_k^\tau$, see Definition \ref{def: Monge-Ampere_measure} for an introduction to this object. 

    Interestingly, this JKO version of the estimate is slightly stronger than what one would expect translating the classical estimate; indeed, by the AM-GM inequality, one has that, for $C^2$ functions, $\frac{1}{d} \Delta u \geq \det(D^2 u)^{1/d}$, hence formally the lower bound on the determinant can be translated into $\Delta p_k^\tau \geq -d \tau^{-1} X_k$ (this can be made rigorous using viscosity solutions, see Lemma \ref{lemma: AM_GM_Monge_Ampere}). Letting $\tau \to 0$, and $k \tau \simeq t$, we recover the Aronson-B\'enilan / Li-Yau estimate using the asymptotic of $(X_k)_{k \geq 1}$ (we refer to \ref{thm: main_result} for a precise statement). On the other hand, a linearization of the determinant as $\tau \to 0$ (expecting $p_k^\tau$ to converge to the pressure of the continuous equation), shows that this estimate is, asymptotically, not better than the Aronson-B\'enilan. 

    Combining this result with the AM-GM like inequality of Lemma \ref{lemma: AM_GM_Monge_Ampere}, and the $L^1$-$L^\infty$ regularization Lemma \ref{lemma: L1-Linfty_regularization}, we obtain the following immediate corollary:

    \begin{corollary}[Local uniform $L^\infty$-bounds on the JKO] \label{corollary: uniform_bounds_JKO}
        For $\bb{B}_r = B_r(x) \subset \Omega$ with $r$ small enough, there exists a constant $M = M(\Omega,t_0,\tau_0,m,r)$ such that for all $\tau \leq \tau_0$, $t \geq t_0$ one has
        \begin{equation*}
            ||\rho_k^\tau||_{L^\infty(\bb{B}_r)} \leq M
        \end{equation*}
    \end{corollary}

    This matches the classical $L^\infty$-regularization effects for the porous-medium, heat, and fast-diffusion equations in the regime $m > m_c^1$. 

    The restriction to small dimension and simple domains stems from our strategy of proof, based on a maximum principle argument: this strategy yields an algebraic matrix inequality involving the Hessian of $u_k^\tau$ at the maximum point. Only in dimension $1$ or $2$ can this inequality be used to derive a lower bound of the determinant of the Hessian. On the other hand, the simple domain assumption is there to be able to handle boundary maximum points. This is handled through an analysis of the behavior of the transport map on the boundary of a cube. An extension to a broader class of domains and to higher dimension would necessitate new ideas. 

\subsection{Structure of the paper}

    \begin{itemize}
        \item In Section $2$, we recall basic results in the theory of optimal transport, functionals over probability measures, and Monge-Amp\`ere measures that will be used in the proof.
        \item In Section $3$, we study the one-step JKO problem, in particular in the regime $m < 1$, and prove existence, uniqueness and optimality conditions for minimizers.
        \item In Section $4$, we show a one-step improvement of Monge-Amp\`ere lower bound on simple domains, under regularity of the initial datum.
        \item In Section $5$, we complete the proof of the Aronson-B\'enilan estimate. 
        \item Finally in Appendix $A$, we show how to obtain $L^1$-$L^\infty$ regularization effects on general domains under sub-harmonic assumptions. 
    \end{itemize}

\subsection{Acknowledgment}

    This work was supported by the European Union via the ERC AdG 101054420 EYAWKAJKOS. The author would like to thank Filippo Santambrogio for suggesting the problem, and for his valuable help in some technical parts of the proof. The author is also grateful to Ivan Gentil for valuable discussions and feedbacks during the preparation of this work. 
    
\section{Preliminaries} \label{section: preliminaries}

    We recall here some basic results in optimal transport, and functionals over the space of probability measures, and Monge-Amp\`ere measures. We refer to the monographs by Villani \cite{Old_New_Villani, Topic_Villani} or Santambrogio \cite{OT_Filippo} for further references. Throughout, $\Omega$ denotes a convex domain, that is a convex subset with non-empty interior, eventually unbounded, of $\bb{R}^d$, or the torus $\bb{T}^d$, and in both cases, $d(x,y)$ is the classical distance on $\Omega$ (Euclidean on subsets of $\bb{R}^d$, and quotient distance on $\bb{T}^d$). By a slight abuse of notations, we shall make no distinction between an absolutely continuous measure and its density with respect to Lebesgue. Similarly, we shall always confuse classes of functions (resp. measures) on the torus and the corresponding class of periodic functions (resp. $\bb{Z}^d$-translation invariant measures). 
    
\subsection{The Wasserstein distance}

    Let $\cl{P}_2(\Omega)$ be the set of all positive measures on $\Omega$ with finite mass and finite second moment $M_2[\mu] := \int_\Omega |x|^2 \dd{\mu}(x) < +\infty$.

    \begin{definition}[Wasserstein distance of order $2$] \label{def: Wasserstein_distance}
        Let $\mu,\nu \in \cl{P}_2(\Omega)$. A transport plan between $\mu$ and $\nu$ is a probability measure $\gamma$ on $\Omega \times \Omega$ with first and second marginals given by $\mu$ and $\nu$. The set of all transport plans between $\mu$ and $\nu$ will be denoted by $\Pi(\mu,\nu)$. The \emph{Wasserstein distance of order $2$} between $\mu$ and $\nu$ is defined as
        \begin{equation} \label{eq: Wasserstein_distance}
            W_2(\mu,\nu)^2 = \min_{\gamma \in \Pi(\mu,\nu)} \int_{\Omega \times \Omega} d(x,y)^2 \dd{\gamma}(x,y)
        \end{equation}
    \end{definition}

    The fact that this is a genuine minimum follows from the direct method in the calculus of variations, and any minimizer is called an optimal transport plan between $\mu$ and $\nu$. It is well-known that $W_2$ is a metric, which metrizes the narrow convergence together with convergence of the second moment, and we shall say that $\mu_n \to \mu$ in $\bb{W}_2(\Omega)$ when $W_2(\mu_n,\mu) \to 0$. 

    A fundamental result in the theory is the so-called Kantorovich duality. 

    \begin{theorem}[Kantorovich duality] \label{thm: Kantorovich_duality}
        Let $\mu,\nu \in \cl{P}_2(\Omega)$. Then one has
        \begin{equation} \label{eq: Kantorovich_duality}
            \frac{1}{2} W_2(\mu,\nu)^2 = \sup \left \{ \int_\Omega \psi \dd{\mu} + \int_\Omega \phi \dd{\nu} \: \middle | \: \psi(x) + \phi(y) \leq \frac{1}{2} d(x,y)^2  \right \} 
        \end{equation}
        Moreover, the supremum is attained at a pair (not necessarily unique) $(\psi,\phi)$ of $c$-conjugate functions, i.e. satisfying
        \begin{equation*}
            \psi(x) = \phi^c(x) = \inf_{y \in \Omega} \left \{ \frac{1}{2} d(x,y)^2 - \phi(y) \right \} \qquad \phi(y) = \psi^c(y) = \inf_{x \in \Omega} \left \{ \frac{1}{2} d(x,y)^2 - \psi(x) \right \}
        \end{equation*}
        Furthermore, if $\gamma$ is an optimal transport plan between $\mu$ and $\nu$, then the inequality $\psi(x) + \phi(y) \leq \frac{1}{2} d(x,y)^2$ is an equality $\gamma$-a.e. Such a pair is called a pair of Kantorovich potentials from $\mu$ to $\nu$.
    \end{theorem}

    The transformation $\psi \mapsto \psi^c$ is usually called the $c$-transform, and functions that are $c$-transform of another function are called $c$-concave functions. It is easy to check that if $\psi$ is $c$-concave, then $\psi = \psi^{cc}$. In the particular setting we are working with, $c$-concavity is equivalent to upper semi-continuity and $1$-semi-concavity. It is worth noticing that, in the case of the torus, by periodicity of the involved functions, one can rewrite the $c$-transform as $\psi^c(x) = \inf_{y \in \bb{R}^d} \{ \frac{1}{2} |x-y|^2 - \psi(y) \}$, which allows to treat both cases with the same definition. The functions $u := \frac{1}{2} |\cdot|^2 - \psi$ and $v := \frac{1}{2} |\cdot|^2 - \phi$ are usually called the Brenier potentials from $\rho$ to $\mu$. Those are convex functions, satisfying $u = v^*$ and $v = u^*$. 

\subsection{Brenier's theorem and Caffarelli's regularity}

    The existence of Kantorovich potentials is the first step in the proof of Brenier's theorem, stating that under absolute continuity assumptions on the densities, the optimal transport plan is in fact an optimal transport map. While this theorem was originally proved by Brenier \cite{Brenier_Theorem} in the Euclidean case, it was extended to Riemannian manifolds, including the torus, by McCann in \cite{Riemannian_Brenier_Theorem}. The case of the torus can ,in fact, be studied independently by identifying probability measures on $\bb{T}^d$ with periodic measures on $\bb{R}^d$. This has been done by Cordero - Erausquin in \cite{Periodic_OT} (in French, see section 1.3.2 of \cite{OT_Filippo} for an English version).

    \begin{theorem}[Brenier - Cordero - McCann] \label{thm: Brenier_thm}
        Let $\mu,\nu \in \cl{P}_2(\Omega)$, and $(\psi,\phi)$ be a pair of Kantorovich potentials from $\mu$ to $\nu$.
        \begin{enumerate}
            \item If $\mu \ll \cl{L}^d$, then $\psi$ is twice differentiable $\mu$-a.e. And, defining the ($\mu$-a.e. defined) map $T(x) := x - \nabla \psi(x)$, one has $T_\# \mu = \nu$, and $(\rm{id}, T)_\# \mu$ is the unique optimal transport plan from $\mu$ to $\nu$. We call $T$ the optimal transport map from $\mu$ to $\nu$.
            \item If we also have $\nu \ll \cl{L}^d$, and if $S$ is the optimal transport map from $\nu$ to $\mu$, then we have $T \circ S = \rm{id}$ (resp. $S \circ T = \rm{id}$) $\nu$-a.e. (resp. $\mu$-a.e.). Furthermore, the Monge-Amp\`ere equation holds $\mu$-a.e. 
            \begin{equation} \label{eq: Monge_Ampere}
                \nu(T(x)) \det(D T(x)) = \mu(x)
            \end{equation}
        \end{enumerate}
    \end{theorem}

    Notice that, in the periodic case, the map $T$ satisfies $T(x+n) = T(x) + n$ for all $n \in \bb{Z}^d$, and therefore defines a map from $\bb{T}^d$ to itself. 

    In general, the optimal potentials are only locally Lipschitz on the support of the measures, and their gradients are of locally bounded variation (using the semi-concavity assumption). In computations however, it is sometimes necessary to assume higher regularity of those functions. This type of regularity can be obtained using the celebrated regularity theory for the Monge-Amp\`ere equation developed by Caffarelli \cite{Caffarelli_Regularity} (we refer to the book \cite{Figalli_Mong_Amp} for an introduction to this deep subject), which roughly states that the transport map is one derivative more regular than the densities. Although originally developed in the Euclidean setting, this can be extended to the torus under the same assumptions as shown by Cordero-Erausquin in \cite{Periodic_OT} (see also \cite{Regu_Monge_Amp_Periodic}).

    Unfortunately, Caffarelli's original theory assumes strong regularity of the boundary on the domain, which the cube does not satisfy. On the other hand, using a reflection-type argument, it was proved by Jhaveri in dimension $2$ in \cite{Regu_Monge_Amp_Cube_Dim2}, and extended by Chen, Liu, and Wang in \cite{Regu_Monge_Amp_Cube_Gen} to other dimension, that one can still obtain some regularity in this case.

    \begin{theorem}[Caffarelli's regularity in torus and cubes] \label{thm: Caff_regu}
        Suppose $\Omega = \bb{T}^d$ or $\Omega = Q = [0,1]^d$. Let $(\psi,\phi)$ be a pair of Kantorovich potentials between two absolutely continuous densities $\mu,\nu$, and suppose that there exists $\epsilon > 0$ such that $\epsilon \leq \nu,\mu \leq \epsilon^{-1}$ a.e.. Then
        \begin{enumerate}
            \item There exists $\beta > 0$ depending only on $\epsilon$ such that $\psi,\phi \in C^{1,\beta}(\Omega)$. Furthermore, $\psi,\phi$ are uniformly $1$-concave.
            \item If for some $\alpha \in (0,1)$ we have $\mu,\nu \in C^{k,\alpha}(\Omega)$ with $k = 0,1$, then $\psi,\phi \in C^{k+2,\alpha}(\Omega)$, and the Monge-Amp\`ere equation holds in the classical sense.
        \end{enumerate}
    \end{theorem} 

    In the torus, one can remove the constraints on $k$, that is, if $\mu,\nu$ are $C^{k,\alpha}(\bb{T}^d)$ for some $k \geq 0$, then the Kantorovich potentials are of class $C^{k+2,\alpha}(\bb{T}^d)$. On the other hand, in the cube, the $C^{3,\alpha}(Q)$ regularity for is sharp, as shown by a counterexample for higher regularity constructed by Jhaveri \cite{Regu_Monge_Amp_Cube_Dim2}. It is worth mentioning that the subject of finding optimal regularity for the transport map in rough domains is a vast topic that has received considerable attention in recent years.

\subsection{Entropy functional} \label{subsection: entropy_functional}

    We introduce the following family of convex functions for $m > 0$.
    \begin{equation*}
        f_m(t) := 
        \begin{cases}
            \frac{1}{m-1} t^m & m \neq 1 \\
            t \log t & m = 1
        \end{cases}
    \end{equation*}
    whose Legendre transform is given by 
    \begin{equation*}
        f_m^*(s) =
        \begin{cases}
            c_m \, [s]_+^{\frac{m}{m-1}}  & \text{if } m > 1 \\[6pt]
            e^{s-1}                        & \text{if } m = 1 \\[6pt]
            \begin{dcases}
                c_m \, (-s)^{\frac{m}{m-1}} & s < 0 \\
                +\infty                      & s \geq 0
            \end{dcases}
                                           & \text{if } m < 1
        \end{cases}
    \end{equation*}
    where $c_m := |m-1|^\frac{1}{m-1}[m^\frac{1}{1-m} + m^\frac{m}{1-m}] > 0$ for $m \neq 1$, and $[s]_+ = \max(s,0)$ is the positive part of $s$. 
    
    For a probability measure $\rho \in \cl{P}_2(\Omega)$, we write the Lebesgue decomposition of $\rho$ with respect to the Lebesgue measure on $\Omega$ as $\rho = \rho^{ac} \cdot \cl{L}^d + \rho^\perp$. 
    
    \begin{definition}[$m$-entropy] \label{def: m_entropy}
        Let $\rho \in \cl{P}_2(\Omega)$. The $m$-entropy of $\rho$ is defined as follows
        \begin{align}
            \mbox{for $m \geq 1$ by } & \scr{E}_m[\rho] = 
            \begin{cases}
                \int_\Omega f_m(\rho^{ac}) \dd{\cl{L}^d} & \mbox{if $\rho^\perp = 0$} \\
                +\infty & \mbox{else}
            \end{cases} \\
            \mbox{for $m < 1$ by } & \scr{E}_m[\rho] = \int_\Omega f_m(\rho^{ac}) \dd{\cl{L}^d}
        \end{align}
    \end{definition}

    The reason for the apparent asymmetry in the definition lies in the different behavior of $f_m$ at $+\infty$. Indeed, for $m \geq 1$, $f_m$ is super-linear (i.e. $t^{-1} f_m(t) \to +\infty$). Whereas for $m < 1$, we have $\lim_{t \to +\infty} t^{-1} f_m(t) = 0$. By standard considerations in the theory of local functionals (\cite[Section 7]{OT_Filippo}), in order to ensure some lower semi-continuity property, one needs to take into account the singular part. 

    \begin{proposition}[Lower semi-continuity of entropy] \label{prop: m_entropy_lsc}
        Suppose that $\rho_n \to \rho$ narrowly in $\cl{P}_2(\Omega)$. Then we have $\scr{E}_m[\rho] \leq \liminf_n \scr{E}_m[\rho_n]$ provided that one of the following conditions holds
        \begin{enumerate}
            \item $\Omega$ is bounded.
            \item $m \geq 1$.
            \item $\Omega$ is unbounded, $m > m_c^2$ and the sequence has uniformly bounded second moment, i.e. $\sup_n M_2[\rho_n] < +\infty$. 
        \end{enumerate}
    \end{proposition}

    \begin{proof}
        For $\Omega$ bounded this follows from \cite[Proposition 7.7]{OT_Filippo}, and the case $m > 1$ follows by positivity of $f_m$ and \cite[Remark 7.8]{OT_Filippo}. 
        
        For the unbounded case, this follows from a simple adaptation of the argument of \cite[Proposition 2.1]{Filippo_Moment_Mes} or \cite[Section 2]{q_moment_mes}. We first need to find a continuous function $b$ such that $(1+|x|)^{-2} |b(x)| \to 0$ as $+\infty$, and such that $f^*_m(-b)$ is integrable. Then, using Jensen's inequality, we have $f_m(t) + f^*_m(-b) + t b \geq 0$, and we deduce that
        \begin{align*}
            \scr{E}_m[\rho] &= -\int_\Omega b \dd{\rho} - \int_\Omega f^*_m(-b) \dd{\cl{L}^d} + \int_\Omega f_m(\rho^{ac}) \dd{\cl{L}^d} + \int_\Omega b \dd{\rho} + \int_\Omega f^*(-b) \dd{\cl{L}^d} \\
            &= -\int_\Omega b \dd{\rho} - \int_\Omega f^*_m(-b) \dd{\cl{L}^d} + \sup_{K \Subset \Omega} \scr{E}_m[\rho \resmes K] + \int_K b \dd{\rho} + \int_K f^*(-b) \dd{\cl{L}^d}
        \end{align*}
        where we used the positivity given by Jensen's inequality to transform the three last terms into the supremum. 
        
        By the previous result on bounded domain, and continuity of $b$, each functionals in the supremum are l.s.c. for the narrow convergence, and the supremum is therefore itself l.s.c. Furthermore, if a sequence converges narrowly and admits uniformly bounded second moment, then the integral of this sequence against $b$ converges as $b$ admits sub-quadratic growth. 
        
        Thus, if one can find such a $b$, then the l.s.c. result follows.
        \begin{enumerate}
            \item For $m = 1$, we can simply take $b(x) = |x|$.
            \item For $0 < m < 1$, we want to take $b(x) = 1 + |x|^q$ for some $0 \leq q < 2$. Then $f_m^*(-b) = c_m (1+|x|^q)^\frac{m}{m-1}$ and the integrability hypothesis is true provided that $\frac{qm}{m-1} < -d$, i.e. $q > \frac{(1-m) d}{m}$. There exists such a $q$ if and only if $2 > \frac{(1-m) d}{m}$, i.e. $m > m_c^2$. \qedhere
        \end{enumerate}
    \end{proof}

    Finally, we will need the following moment estimate for the entropy, which will be used to derive coercivity in the one-step JKO scheme problem. 

    \begin{proposition}[Lower bound on entropy] \label{prop: entropy_lower_bound}
        Suppose $m > m_c^2$. Then there exists a constant $C(m,d) \in \bb{R}$ such that for all $\rho \in \cl{P}_2(\bb{R}^d)$ we have
        \begin{equation*}
            \scr{E}_m[\rho] \geq 
            \begin{cases}
                C(m,d) M_2[\rho]^{\frac{d}{2}(1-m)} & m \neq 1 \\
                C(1,d) - \frac{d}{2} \log M_2[\rho] & m = 1
            \end{cases}
        \end{equation*}
        Moreover $C(m,d)$ can be obtained by solving the variational problem
        \begin{equation*}
            C(m,d) = \inf_{\rho \in \cl{P}_2(\bb{R}^d), M_2[\rho] = 1} \scr{E}_m[\rho]
        \end{equation*}
    \end{proposition}

    \begin{proof}
        We consider the following optimization problem for $m > m_c^2$:
        \begin{equation*}
            E_m(M) = \inf_{\rho \in \cl{P}_2(\bb{R}^d), M_2[\rho] = M} \scr{E}_m[\rho]
        \end{equation*}
        which is finite, as $\scr{E}_m[\rho] \geq 0$ for $m > 1$, and using Fenchel's inequality $f_m(t) \geq -(1+|x|^2) t + f_m^*(-(1+|x|^2))$ for $m \leq 1$. By approximation, we can restrict the optimization to $\rho \ll \cl{L}^d$. If $\rho$ is admissible, then $\eta = M^\frac{d}{2} \rho(\sqrt{M} \cdot)$ satisfies $\eta \in \cl{P}_2(\bb{R}^d)$, $M_2[\eta] = 1$, $\scr{E}_m[\eta] = M^{\frac{d}{2}(m-1)} \scr{E}_m[\rho]$ for $m \neq 1$, and $\scr{E}_1[\eta] = \scr{E}_1[\rho] + \frac{d}{2} \log M$ for $m = 1$. This rescaling implies that
        \begin{equation*}
            E_m(M) = 
            \begin{cases}
                E_m(1) M^{\frac{d}{2}(1-m)} & m \neq 1 \\
                E_m(1) - \frac{d}{2} \log M & m = 1
            \end{cases}
        \end{equation*}
        and we conclude using $\scr{E}_m[\rho] \geq E_m(M_2[\rho])$ for any $\rho \in \cl{P}_2(\bb{R}^d)$. 
    \end{proof}

    \begin{remark}
        One can prove that optimizers for $E_m(M)$ exist and are Gaussian for $m = 1$, and of Barenblatt form for $m \neq 1$, i.e. of the form $(A - B |x|^2)_+^\frac{1}{m-1}$ for $m > 1$ and $(A+B |x|^2)^\frac{1}{m-1}$ for $m < 1$ with $A,B > 0$. Furthermore, one can then express $E_m(1)$ in terms of Gamma-type functions (for example, $E_1(1) = \frac{d}{2} \log(2 \pi d)$, attained for the standard Gaussian).   
    \end{remark}

\subsection{Monge-Amp\`ere measure}

    The Monge-Amp\`ere measure is a weak extension of the non-linear second order quantity $\det(D^2 u)$, to an arbitrary convex function $u$. It allows one to define such a notion even if the function $u$ is not regular, and is the cornerstone of the regularity theory of the Monge-Amp\`ere equation. A detailed introduction can be found in \cite{Figalli_Mong_Amp} or \cite{Monge_Ampere_Nam_Qe_Le}. 

    \begin{definition}[Monge-Amp\`ere measure] \label{def: Monge-Ampere_measure}
        Let $\Omega$ be a convex domain of $\bb{R}^d$, $u : \Omega \to \bb{R}$ a convex function. The Monge-Amp\`ere measure of $u$ is the measure on $\rm{int}(\Omega)$ defined by
        \begin{equation*}
            \mu_u(E) := |\partial u[E]| \qquad \partial u[E] = \bigcup_{x \in E} \partial u(x)
        \end{equation*}
    \end{definition}

    The fact that this defines a genuine Borel measure is a non-trivial fact in the theory, see \cite[Theorem 2.3]{Figalli_Mong_Amp} for a proof. If $u$ is of class $C^2$ (or even merely $C^{1,1}$), then this measure coincides with $\det(D^2 u) \cdot \cl{L}^d$. As is usual with weak notions, this measure admits better stability than the non-linear object $\det(D^2 u)$.

    \begin{proposition}[Stability of Monge-Amp\`ere measure, Proposition 2.6 \cite{Figalli_Mong_Amp} ] \label{prop: stability_Monge_Ampere}
        Suppose that $u_n \to u$ locally uniformly on $\rm{int}(\Omega)$. Then $\mu_{u_n} \rightharpoonup \mu_u$ in the weak-$*$ topology (in duality with $C_c(\rm{int}(\Omega))$). 
    \end{proposition}
    
    We will say that an inequality of the form $\det(D^2 u) \geq \lambda$ holds in the Monge-Amp\`ere sense if one has $\mu_u \geq \lambda \cdot \cl{L}^d$ in the sense of measures. We shall also sometimes write $\det(D^2 u)^{1/d} \geq \lambda$ as a shorthand for $\det(D^2 u) \geq \lambda^d$. The AM-GM inequality $\Delta u \geq d \cdot\det(D^2 u)^{1/d}$ can be extended to obtain a sub-harmonic bound from a Monge-Amp\`ere lower bound.

    \begin{lemma}[Sub-harmonicity from Monge-Amp\`ere lower bound] \label{lemma: AM_GM_Monge_Ampere}
        Suppose that $\det(D^2 u)^{1/d} \geq \lambda$ in the Monge-Amp\`ere sense. Then $\Delta u \geq d \cdot \lambda$ in the viscosity/weak sense. 
    \end{lemma}

    \begin{proof}
        By \cite[Proposition 7.7]{Monge_Ampere_Nam_Qe_Le}, the inequality $\det(D^2 u) \geq \lambda^d$ also holds in the viscosity sense in $\rm{int}(\Omega)$. Furthermore, by \cite[Theorem 7.2]{Monge_Ampere_Nam_Qe_Le}, convexity of $u$ implies that we have $\lambda_{min}(D^2 u) \geq 0$ in the viscosity sense. We argue that this implies the asserted result. Indeed, fix $\psi \in C^2(\bb{R}^d)$ such that $\psi - u$ admits a local minimum at some point $x_0 \in \rm{int}(\Omega)$, then combining the two viscosity inequalities, we have that $D^2 \psi(x_0)$ is symmetric semi-definite positive, and $\det(D^2 \psi(x_0)) \geq \lambda^d$. But by the AM-GM inequality, for any symmetric semi-definite positive matrix $N$, we have $\Tr N \geq d \cdot \det(N)^{1/d}$, therefore we have
        \begin{equation*}
            \Delta \psi(x_0) = \Tr D^2 \psi(x_0) \geq d \cdot \det(D^2 \psi(x_0))^\frac{1}{d} \geq d \cdot \lambda
        \end{equation*}
        concluding the proof by the arbitrariness of $\psi$ and $x_0$. 
    \end{proof}

\subsection{A $L^1-L^\infty$-regularization Lemma}

    Doing approximation, we shall need a $L^1-L^\infty$ regularization Lemma for functions admitting some sub-harmonic lower bound. We define the function $h_m$ by
    \begin{equation*}
        h_m(z) := 
        \begin{cases}
            z^{m-1} & m > 1 \\
            \log z & m = 1 \\
            -z^{m-1} & m < 1
        \end{cases}
    \end{equation*}

    \begin{lemma}[$L^1$-$L^\infty$ regularization effects] \label{lemma: L1-Linfty_regularization}
        Suppose $m > m_c^1$. Let $\bb{B}_2 = B_{2r}(x_0)$ be some ball, and let $\bb{B}_1 = B_r(x_0)$. Then if $g \in L^1_+(\bb{B}_2)$, with $||g||_{L^1(\bb{B}_{2})} \leq 1$, is such that $\Delta h_m(g) \geq -K$ weakly on $\bb{B}_2$ (assuming $g > 0$ if $m \leq 1$). Then there exists constants $r_*(d,m,K)$ and $M(r,d,m,K)$ such that if $r < r_*$ then $g \in L^\infty(\bb{B}_1)$ with $||g||_{L^\infty(\bb{B}_1)} \leq M$. 
    \end{lemma}

    \begin{proof}
        The case $m > 1$ follows from \cite[Lemma A.3]{porous-medium}. In the case $m_c^1 < m \leq 1$, $h_m$ is increasing convex. Let $y \in \bb{B}_1$, then $B_r(y) \subset \bb{B}_2$, and using the sub-harmonicity of $h_m(g) + \frac{K}{2} |x-y|^2$ we have
        \begin{align*}
            h_m(g(y)) &\leq \fint_{B_r(y)} h_m(g) \dd{\cl{L}^d} + \frac{K}{2} \fint_{B_r(0)} |x|^2 \dd{x} \\
            &= \fint_{\bb{B}_1} h_m(g) \dd{\cl{L}^d} + c_d K r^2
        \end{align*}
        On the other hand, using concavity of $h_m$ and its monotony we can bound
        \begin{equation*}
            \fint_{B_r(y)} h_m(g) \dd{\cl{L}^d} \leq h_m \left ( \fint_{B_r(y)} g \dd{\cl{L}^d} \right ) \leq h_m \left ( \frac{1}{r^d \omega_d} ||g||_{L^1(\bb{B}_2)} \right ) 
        \end{equation*}
        where $\omega_d$ is the volume of $B_1(0)$. Combining the two bounds, and using that $||g||_{L^1(\bb{B}_2} \leq 1$, we obtain
        \begin{equation*}
            \begin{cases}
                \log g(y) \leq - \log \omega_d r^d + c_d K r^2 & m = 1 \\
                -g^{m-1}(y) \leq -\omega_d^{1-m} r^{d(1-m)}  + c_d K r^2 & m < 1 
            \end{cases}
        \end{equation*}
        In the first case, taking the exponential give the asserted $L^\infty(\bb{B}_1)$-bound. On the other hand, to get a $L^\infty$-bound in the second case, we need to ensure that the right-hand-side is negative. This is true provided that $r^{2 + d(m-1)} \leq C(d,m) K^{-1}$ for some constant $C(d,m)$, which concludes. 
    \end{proof}

\section{The JKO Scheme} \label{section: JKO_Scheme}

    The JKO scheme consists of iterating the following minimization problem:
\begin{equation}
    \rho \in \argmin_{\eta \in \cl{P}_2(\Omega)} \scr{E}_m[\eta] + \frac{1}{2 \tau} W_2^2(\eta,\mu)
\end{equation}
where ,throughout this section, we assume that $m > m_c^1$, and additionaly that $m > m_c^2$ if $\Omega$ is unbounded. As explained in the introduction, to handle the case $m < 1$, one would like to prove that minimizers are always absolutely continuous, unique, and characterize such minimizers by their optimality conditions, at least on bounded domains. 

For unbounded domain, we will focus on domains admitting some boundary regularity: we shall say that an convex domain $\Omega \subset \bb{R}^d$ is volume-regular if there exists a constant $V > 0$ such that $|\Omega \cap B_r(x)| \geq V r^d$ for all $x \in \Omega$, $r \leq \rm{diam}(\Omega)$. Any bounded convex domain is in fact volume regular, but for unbounded domain, one need to ensure for instance some uniform-Lipschitz regularity of the boundary to ensure that this is the case (note that convex domains has locally Lipschitz boundary, but the Lipschitz constant might blow up far from the origin for unbounded domains). We will see that this condition is enough to ensure uniqueness and absolutely continuity of minimizers. 

\subsection{Existence and Qualitative properties}
    Existence readily follows from the direct method using the results of the previous section. Uniqueness of minimizers is more involved: for $m \geq 1$ it follows immediately from the strict convexity of $\scr{E}_m$ and convexity of $W_2^2(\cdot,\mu)$. On the other hand, for $m < 1$, the entropy is no longer strictly convex, but it admits a weaker form of strict convexity, which, combined with optimality conditions, still gives uniqueness of minimizers. 

    \begin{proposition}[Existence of minimizers] \label{prop: existence_one_step}
        For any $\mu \in \cl{P}_2(\Omega)$, and additionally assume $\Omega$ to be volume-regular if $m < 0$. Then there exists a unique minimizer for the one-step JKO problem starting from $\mu$. We will denote by $Q_m^\tau[\mu]$ this unique minimizer. 
    \end{proposition}
    
    \begin{proof}[Proof of Proposition \ref{prop: existence_one_step}]
        We divide the proof into existence and uniqueness.
        \begin{itemize}
            \item \textbf{Existence}: In bounded domains, this follows immediately from l.s.c. of the entropy and Wasserstein distance (\cite[Proposition 7.4]{OT_Filippo}) and from compactness of sequences of probability measures for narrow convergence on such sets. We now focus on the case where $\Omega$ is unbounded (and therefore $m > m_c^2$). Let $\rho \in \cl{P}_2(\Omega)$. Integrating the inequality $\frac{1}{2} |x-y|^2 \geq \frac{1}{4} |x|^2 - \frac{1}{4} |y|^2$ against any optimal transport plan from $\rho$ to $\mu$, we obtain $\frac{1}{2 \tau} W_2^2(\rho,\mu) \geq \frac{1}{4 \tau} M_2[\rho] - \frac{1}{4 \tau} M_2[\mu]$. Combined with Proposition \ref{prop: entropy_lower_bound} this gives
            \begin{equation} \label{eq: lower_bound_JKO}
                \scr{E}_m[\rho] + \frac{1}{2 \tau} W_2^2(\rho,\mu) \geq \begin{cases}
                    \frac{1}{4 \tau} M_2[\rho] + C(m,d) M_2[\rho]^{\frac{d}{2}(1-m)} - \frac{1}{4 \tau} M_2[\mu] & m \neq 1 \\
                    \frac{1}{4 \tau} M_2[\rho] - \frac{d}{2} \log M_2[\rho] - C(1,d) - \frac{1}{4 \tau} M_2[\mu] & m = 1
                \end{cases}
            \end{equation}
            This implies that any minimizing sequence has uniformly bounded second moment, which gives narrow compactness of such sequences by Prokhorov's theorem. Combining this with Proposition \ref{prop: m_entropy_lsc} we can again use the direct method to conclude. 
    
            \item \textbf{Uniqueness}: For $m \geq 1$, this follows immediately from strict convexity of $\scr{E}_m$ and convexity of $W_2^2(\cdot,\mu)$. For $m < 1$, we only have the following weaker version of strict convexity: if $\rho,\eta$ are such that $\scr{E}_m[t \rho + (1-t) \eta] = t \scr{E}_m[\rho] + (1-t) \scr{E}_m[\eta]$ for some $t \in (0,1)$, then $\rho^{ac} = \eta^{ac}$. In particular, uniqueness holds if there exists at least one absolutely continuous minimizer. As this is the case if $\Omega$ is volume-regular by Corollary \ref{coro: absolute_continuity_unbounded}, we can conclude on the uniqueness. \qedhere
        \end{itemize}
    \end{proof}

\subsection{Optimality conditions}

    Before deriving the optimality conditions, let's prove a qualitative behavior in the case $m \leq 1$. 

    \begin{proposition}[Positivity and integrability of optimizers] \label{prop: positivity_integrability}
        Let $\mu \in \cl{P}_2(\Omega)$, $\rho$ a minimizer for the one-step JKO problem starting from $\mu$. If $m \leq 1$, then $\rho^{ac} > 0$ a.e. and $f_m'(\rho^{ac}) \in L^1_{loc}(\Omega)$
    \end{proposition}

    \begin{proof}
        We closely follow the proof of \cite[Lemma 8.6]{OT_Filippo} with minor modifications. Let $\xi \in \cl{P}_{2,ac}(\Omega)$ be a constant density if $\Omega$ bounded, and in the case of $\Omega$ unbounded, proportional to $e^{-|x|^2}$ for $m = 1$, and to $(1+|x|^2)^\frac{1}{m-1}$ for $m < 1$. It satisfies $\scr{E}_m[\xi] < +\infty$ and $|f'_m(\xi)| \leq A(1+|x|^2)$ for some constant $A > 0$.
        
        Let $\rho$ be a minimizer, with absolutely continuous part $g$, and set $\rho_\epsilon = \epsilon \xi + (1-\epsilon) \rho$ whose absolute continuous part is $g_\epsilon = \epsilon \xi + (1-\epsilon) g$. We will first prove that $g > 0$ a.e., then we will derive the integrability of $f'_m(g) \xi$, which will conclude. 
        \begin{itemize}
            \item Using convexity of $W_2^2(\cdot,\mu)$ and optimality of $\rho$, if $\Sigma = \{g = 0 \}$ we have
            \begin{equation*}
                -\int_{\Sigma} f_m(\epsilon \xi) \dd{\cl{L}^d} + \int_{\Omega \setminus \Sigma} (f_m(g) - f_m(g_\epsilon)) \dd{\rho} \leq \frac{\epsilon}{2 \tau}(W_2^2(\xi,\mu) - W_2^2(\rho,\mu))
            \end{equation*}
            By convexity of $f_m$, we have $f_m(g) - f_m(g_\epsilon) \geq f_m'(g_\epsilon)(g-g_\epsilon) = \epsilon f_m'(g_\epsilon)(g - \xi)$. However, using the monotonicity of $f_m'$, we have
            \begin{equation*}
                f_m'(g_\epsilon)(g-\xi) = \frac{1}{1-\epsilon} f_m'(g_\epsilon)(g_\epsilon - \xi) \geq \frac{1}{1-\epsilon} f_m'(\xi)(g_\epsilon - \xi) = f_m'(\xi)(g-\xi)
            \end{equation*}
            By hypothesis $|f_m'(\xi)| \leq A(1+|x|^2)$, which implies that $f_m'(\xi)(g-\xi) \in L^1(\Omega)$. Integrating, we obtain
            \begin{equation*}
                - \int_{\Sigma} f_m(\epsilon \xi) \dd{\cl{L}^d} \leq \epsilon \left ( \int_{\Omega} f_m'(\xi)(\xi-g) \dd{\cl{L}^d} + \frac{1}{2 \tau} W_2^2(\xi,\mu) - \frac{1}{2 \tau} W_2^2(\rho,\mu) \right ) 
            \end{equation*}
            Therefore $\epsilon^{-1} \int_{\Sigma} f_m(\epsilon \xi) \dd{\cl{L}^d}$ is bounded from below. 
            \begin{itemize}
                \item For $m = 1$, this is equal to $\log(\epsilon) \int_\Sigma \xi \dd{\cl{L}^d} + \epsilon \int_\Sigma \xi \log(\xi) \dd{\cl{L}^d}$ which converges to $-\infty$ if $|\Sigma| \neq 0$ as $\xi > 0$.
                \item For $m < 1$, this is equal to $\frac{m}{m-1} \epsilon^{m-1} \int_{\Sigma} \xi^m \dd{\cl{L}^d}$ which again converges to $-\infty$ if $|\Sigma| \neq 0$ as $\xi > 0$.  
            \end{itemize}
            Hence we have $|\Sigma| \neq 0$ and $g > 0$ a.e. 
            \item Rewriting the previous inequalities using $g > 0$ a.e. we had
            \begin{align*}
                &\int_\Omega f_m'(g_\epsilon)(g-\xi) \dd{\cl{L}^d} \leq \frac{1}{2 \tau} W_2^2(\xi,\mu) - \frac{1}{2 \tau} W_2^2(\rho,\mu) \\
                &f_m'(g_\epsilon)(g-\xi) \geq f_m'(\xi)(g-\xi) \in L^1(\Omega)
            \end{align*}
            Applying Fatou's lemma to the positive part, and dominated convergence to the negative part, using the second inequality. We can pass to the limit $\epsilon \to 0$ in both inequalities, which provides the integrability of $f'_m(g)(g-\xi)$. Finally, as $f'_m(g) g = m f_m(g) \in L^1(\Omega)$ by $-\infty < \scr{E}_m[\rho] < +\infty$, we conclude on the integrability of $f'_m(g) \xi$. \qedhere
        \end{itemize}
    \end{proof}

    \begin{remark}
        In bounded domains, one shall be able, using the results of \cite[Appendix B]{cournot_nash_Inada}, to derive that $\rho^{ac} \geq \delta(m,\tau,\Omega) > 0$. This is due to the fact that $f_m$ satisfies the lower Inada condition $f'_m(0) = -\infty$. Similarly, for $m \geq 1$, one can obtain $\rho^{ac} \leq M(m,\tau,\Omega)$ by the upper Inada condition $f'_m(+\infty) = +\infty$. 
    \end{remark}
    
    In the super-linear case $m \geq 1$, the optimality conditions in the JKO scheme are pretty well understood. We refer for example to \cite[Section 7.4.1]{OT_Filippo} or to \cite{L1_contraction}. For $m < 1$, we will follows ideas used by Khanh and Santambrogio in the context of $q$-moment measure in \cite{q_moment_mes} to show that a minimizer can't have a singular part. 

    \begin{theorem}[Optimality conditions] \label{thm: optimality_conditions}
        Let $\Omega$ bounded, and $\mu \in \cl{P}_2(\Omega)$. Let $\rho$ be a minimizer of the one-step JKO from $\mu$. Then there exists a tuple of Kantorovitch potentials $(\psi,\phi)$ from $\rho$ to $\mu$ such that:
        \begin{enumerate}
            \item $m = 1$: We have 
            \begin{equation*}
                \tau \log \rho = -\psi
            \end{equation*}
            \item $m > 1$: We have
            \begin{equation*}
                \tau \frac{m}{m-1} \rho^{m-1} = [-\psi]_+
            \end{equation*}
            where $[z]_+ = \max(z,0)$. In particular, $\tau f'_m(\rho) + \frac{1}{2} |x|^2 = \max(\frac{1}{2} |x|^2 - \psi, \frac{1}{2} |x|^2)$ is convex. 
            \item $m_c^1 < m < 1$: $\rho$ is absolutely continuous, $\psi < 0$, and we have 
            \begin{equation*}
                \tau \frac{m}{m-1} \rho^{m-1} = -\psi
            \end{equation*}
        \end{enumerate}
    \end{theorem}

    \begin{proof}
        The only new result is the last case. We divide into several steps:
        \begin{itemize}
            \item \textbf{Step 1 - Directional derivative inequality:} We say that a measure $\chi \in \cl{P}_2(\Omega)$ is admissible if $f_m'(\rho^{ac}) \chi^{ac} \in L^1(\Omega)$ and $\scr{E}_m[\chi] < +\infty$. Note that if $\chi^{ac}$ is bounded, then $\chi$ is admissible by Proposition \ref{prop: positivity_integrability}. Fix such a $\chi$, and let $\rho_\epsilon := \epsilon \chi + (1-\epsilon) \rho$. Then by dominated convergence theorem, $\scr{E}_m[\rho_\epsilon]$ is differentiable at $0$ with
            \begin{equation*}
                \dv{\epsilon}_{|\epsilon = 0} \scr{E}_m[\rho_\epsilon] = \int_\Omega f'_m(\rho^{ac})(\chi^{ac}-\rho^{ac}) \dd{\cl{L}^d} 
            \end{equation*}
            On the other hand, since $\rho$ is supported on $\Omega$, then by \cite[Proposition 7.17-7.18]{OT_Filippo}, the Kantorovich potentials $(\psi,\phi)$ from $\rho$ to $\mu$ are unique up to translation, and $W_2^2(\rho_\epsilon,\mu)$ is differentiable at $0$ with
            \begin{equation*}
                \dv{\epsilon}_{|\epsilon = 0} \frac{1}{2} W_2^2(\rho_\epsilon,\mu) = \int_\Omega \psi \dd{[\chi-\rho]}
            \end{equation*}
            Since, by optimality of $\rho$ we have $ \dv{\epsilon}_{|\epsilon = 0} \tau \scr{E}_m[\rho_\epsilon] + \frac{1}{2} W_2^2(\rho_\epsilon,\mu) \geq 0$, and dividing the derivative of the Wasserstein distance into the absolutely continuous and singular part, we deduce that
            \begin{equation} \label{eq: integral_optimality_conditions}
                \int_\Omega (\tau f_m'(\rho^{ac}) + \psi)(\rho^{ac} - \chi^{ac}) \dd{\cl{L}^d} + \int_\Omega \psi \dd{[\chi^\perp - \rho^\perp]} \geq 0
            \end{equation}

            \item \textbf{Step 2 - Pointwise optimality condition:} Define $C := \essinf(\tau f'_m(\rho^{ac}) + \psi)$ and $C' := \inf \psi$, well defined by continuity of $\psi$. If we take $\chi^\perp = \rho^\perp$ in \ref{eq: integral_optimality_conditions}, and using the argument of \cite[Proposition 7.20]{OT_Filippo}, we get that $\tau f'_m(\rho^{ac}) + \psi = C$ a.e. on $\Omega$. On the other hand, if we let $\chi^{ac} = \rho^{ac}$, then taking $\chi^\perp$ concentrated on the set $\{ \psi = C' \}$ shows that $\rho^\perp$ is concentrated on this set. Therefore we have
            \begin{equation*}
                \begin{cases}
                    \tau f'_m(\rho^{ac}) + \psi = C & \mbox{a.e.} \\
                    \psi = C' & \mbox{$\rho^\perp$ a.e.}
                \end{cases}
            \end{equation*}

            \item \textbf{Step 3 - Equality of constants under existence of singular part:} Suppose $\rho^\perp \neq 0$, let $t := \rho^{ac}(\Omega)$ and $s := \rho^\perp(\Omega)$, so that $t+s = 1$. Let $a,b \geq 0$ such that $a t + b s = 1$, then the measure $\chi = a \rho^{ac} \cdot \cl{L}^d + b \rho^\perp$ is admissible. Putting this measure into equation \ref{eq: ineq_optimal_point} we get $C t (a-1) + C' s (b-1) \geq 0$. As $s > 0$, solving for $b$ gives $b = \frac{1-at}{s}$ provided that $0 \leq a \leq 1/t$. Replacing $b$ by this value, and $s$ by $1-s$ we obtain
            \begin{equation*}
                C t (a-1) + C' s (b-1) = (a-1) t (C-C') \geq 0
            \end{equation*}
            Since $1/t > 1$, $a-1$ can take both positive and negative value, therefore we must have $C=C'$. In particular, if $\rho^\perp$ is non-zero, the conditions become
            \begin{equation*}
                \begin{cases}
                    \tau f'_m(\rho^{ac}) + \psi = C & \mbox{a.e.} \\
                    \psi = C & \mbox{$\rho^\perp$ a.e.}
                \end{cases}
            \end{equation*}
            
            \item \textbf{Step 4 - Absolute continuity of $\rho$:} Suppose that $\rho^\perp \neq 0$, then $C = C'$. Let $x_0$ be a minimum for $\psi$, then by Lemma \ref{lemma: quadratic_deviation_min} below, we have
            \begin{equation*}
                \tau f_m'(\rho^{ac}(x)) = C - \psi(x) \geq -\frac{1}{2} |x-x_0|^2
            \end{equation*}
            Hence
            \begin{equation*}
                \rho^{ac}(x) \geq \left ( \frac{|m-1|}{\tau m} \right )^\frac{1}{m-1} |x-x_0|^\frac{2}{m-1}
            \end{equation*}
            As convex domains satisfy interior cone condition, we can find a cone 
            \begin{equation*}
                C(\nu,\theta,h) = \{ t v, t \leq h, v \in \bb{S}^d, |v \cdot \nu| \leq \theta \}
            \end{equation*}
            such that $x + C(\nu,\theta,h) \subset \Omega$. Indeed, consider a ball $\overline{B}_F(x_0,r) \subset \Omega$, then the set $\rm{conv}(x,\overline{B}_F(x_0,r))$ contains such a set.
            Integrating the previous inequality over this set we have:
            \begin{align*}
                1 \geq \int_{x_0 + C(\nu,\theta,h)} \rho^{ac} \dd{\cl{L}^d} \geq \left ( \frac{|m-1|}{\tau m} \right )^\frac{1}{m-1} \int_{C(\nu,\theta,h)} |x|^\frac{2}{m-1} \dd{x} \\
                = \left ( \frac{|m-1|}{\tau m} \right )^\frac{1}{m-1} \cl{H}^{d-1}(\Sigma_{\beta,\nu}) \int_0^h r^{d-1 + \frac{2}{m-1}} \dd{r} 
            \end{align*}
            where $\Sigma_{\nu,\beta} := \{ v \in \bb{S}^d, |v \cdot \theta| \leq \theta \}$. The last term being infinite as $m > m_c^1$, we get a contradiction. Therefore $\rho$ is absolutely continuous, concluding the proof up to replacing $\psi$ by $\psi - C$. \qedhere
        \end{itemize}
    \end{proof}

    \begin{lemma}[Quadratic deviation from minimum] \label{lemma: quadratic_deviation_min}
        Let $\Omega$ be bounded convex, let $(\psi,\phi)$ be a pair of Kantorovitch potentials between two measures $\rho,\eta$ on $\Omega$. Let $x_0 \in \Omega$ be a minimum point of $\psi$ (which exists by continuity), then for all $x \in \Omega$ we have
        \begin{equation}
            \psi(x) \leq \psi(x_0) + \frac{1}{2} |x-x_0|^2
        \end{equation}
    \end{lemma}

    \begin{proof}
        Define $u(x) := \frac{1}{2} |x|^2 - \psi$. By Kantorovich duality, we have $u(x) = \max_{y \in \Omega} x \cdot y - u^*(y)$, hence $\partial u[x_0]$ is non-empty and contains a point in $\Omega$. Let $p$ be such a point. By the sub-differential inequality we have
            \begin{equation*}
                u(x) \geq u(x_0) + p \cdot (x-x_0)
            \end{equation*}
            But as $u(x) \leq \frac{1}{2} |x|^2 - \psi(x_0) = \frac{1}{2} |x|^2 - \frac{1}{2} |x_0|^2 + u(x_0)$ we get 
            \begin{equation*}
                \frac{1}{2} |x|^2 - \frac{1}{2} |x_0|^2 + p \cdot (x_0 - x) = \frac{1}{2} (2 p - x_0 - x) \cdot (x_0-x) \geq 0 
            \end{equation*}
            for all $x \in \Omega$. In particular, if we take $x = p$, we get $-|p-x_0|^2 \geq 0$, therefore $p = x_0$. We deduce that
            \begin{align*}
                \psi(x) &= \frac{1}{2} |x|^2 - u(x) \leq \frac{1}{2} |x|^2 - u(x_0) - x_0 \cdot (x-x_0) \\
                &= \psi(x_0) + \frac{1}{2} |x|^2 - \frac{1}{2} |x_0|^2 - x_0 \cdot (x-x_0) = \psi(x_0) + \frac{1}{2} |x-x_0|^2
            \end{align*}
            Therefore $\psi$ deviates at most quadratically from its minimum. 
    \end{proof}

    As a first consequence of the optimality conditions, and by approximation, one can derive that there exists at least an optimizers are always absolutely continuous, even in the unbounded case.

    \begin{corollary}[Absolute continuity in unbounded domains] \label{coro: absolute_continuity_unbounded}
        Suppose $\Omega$ is a volume-regular domain, $m_c^2 < m < 1$. Then there exists a least one absolutely continuous minimizer for the one-step JKO problem.
    \end{corollary}

    \begin{proof}
        In the bounded case, this follows immediately from the optimality conditions. We now consider $\Omega$ unbounded volume-regular domain (with constant $V$), and set $\Omega_N := \Omega \cap B_N(0)$, which is, for $N$ large enough, volume-regular with constant $\frac{1}{2} V$. We approximate $\mu$ by a sequence $\mu_N \in \cl{P}_2(\Omega_N)$ converging in $\bb{W}_2(\Omega)$ to $\mu$, and we let $(\rho_N)_N$ be the corresponding sequence of minimizers. By Proposition \ref{prop: stability_JKO}, up to subsequence, $\rho_N \to \rho$ in $\bb{W}_2(\Omega)$ where $\rho$ is a minimizer of the one-step JKO scheme starting from $\mu$. 

        By the optimality conditions, for each $N$, there exists a Kantorovitch potential $\psi_N$ such that $\tau \frac{m}{m-1} \rho_N^{m-1} = -\psi_N$. We argue that this implies uniform $L^\infty$-bound on $\rho_N$. Indeed, consider $x_0$ maximum point of $\rho_N$, or equivalently, minimum point of $\psi_N$. By Lemma \ref{lemma: quadratic_deviation_min}, for all $x \in \Omega_N$, $\tau \frac{m}{m-1} \rho_N^{m-1}(x_0) \leq \tau \frac{m}{m-1} \rho_N^{m-1}(x) - \frac{1}{2} |x-x_0|^2$. Integrating over $B(x_0,r) \cap \Omega_N$ for $r \leq \rm{diam}(\Omega_N)$ and using monotony and convexity of $-t^{m-1}$, as in the proof of Lemma \ref{lemma: L1-Linfty_regularization}, we obtain
        \begin{align*}
            \tau \frac{m}{m-1} ||\rho_N||_\infty^{m-1} &\leq \tau \frac{m}{m-1} \frac{1}{|\Omega_N \cap B(x_0,r)|^{m-1}} - c_d \frac{r^{d+2}}{|\Omega_N \cap B(x_0,r)|} \\
            &\leq c_1(m,\tau,V) r^{d(1-m)} - c_2(m,V) r^2
        \end{align*}
        taking $r$ small enough, depending only on $m,\tau,V$, the right-hand-side is negative, and we can inverse the relation to obtain $||\rho_N||_{L^\infty(\Omega_N)} \leq M(\tau,m,V)$ uniformly in $N$. Combining this with the $\bb{W}_2(\Omega)$-convergence implies that the limit $\rho$ is absolutely continuous (and even $L^\infty(\Omega)$). 
    \end{proof}

    Finally, we have the following propagation of upper and lower bound. It follows either by a maximum principle type argument, in the spirit of \cite[Proposition 7.32]{OT_Filippo}, or by the comparison principle, either using the $L^1$-contraction principle obtained by Jacobs, Kin, and Tong in \cite{L1_contraction} in the super-linear case, or by the general theory developed by L\'eger and Sylvestre in \cite{Comparison_Maxime}, which can be adapted to the non super-linear case once uniqueness is ensured. 

    \begin{proposition}[Proposition of upper and lower bound] \label{prop: propagation_upper_lower_bound}
        Suppose that $\epsilon \leq \mu \leq \epsilon^{-1}$, then the same holds for $Q_m^\tau[\mu]$.  
    \end{proposition}

\subsection{Stability of the JKO scheme}

    In this section, we show a simple stability result when both the domain, and the initial data, are approximated. More precisely, consider the following framework:
    \begin{itemize}
        \item We let $(\Omega_n)_{n \geq 0}$ be a non-decreasing sequence of convex domains of $\bb{R}^d$ (resp. $\Omega_n = \bb{T}^d$ for all $n \geq 0$), and we set $\Omega = \bigcup_{n \geq 0} \Omega_n$. 
        \item For each $n \geq 0$, we let $\mu_n \in \cl{P}_2(\Omega_n)$, and we consider $\mu \in \cl{P}_2(\Omega)$. Furthermore we assume that $\mu_n \to \mu$ in $\bb{W}_2(\Omega)$. We let $\rho_n$ be a minimizer for the one-step JKO problem starting from $\mu_n$ on the domain $\Omega_n$. 
    \end{itemize}
    We then have the following stability result, which follows by a simple $\Gamma$-convergence argument. 

    \begin{proposition}[Stability] \label{prop: stability_JKO}
        Under the above framework, then up to subsequence we have $\rho_n \to \rho$ in $\bb{W}_2(\Omega)$ as $n \to +\infty$, where $\rho$ is a minimizer of the one-step JKO problem starting from $\mu$ on the domain $\Omega$. 
    \end{proposition}

    \begin{proof}
        We will first prove that the convergence holds narrowly by a $\Gamma$-convergence argument. Then using again this $\Gamma$-convergence, we will improve this narrow convergence to full $\bb{W}_2(\Omega)$-convergence. 
        \begin{itemize}
            \item \textbf{$\Gamma$-convergence of JKO functional}:
            We define the functional:
            \begin{equation*}
                \scr{J}_n[\eta] = 
                \begin{cases}
                    \scr{E}_m[\eta] + \frac{1}{2 \tau} W_2^2(\eta,\mu_n) & \mbox{if $\eta \in \cl{P}_2(\Omega_n)$} \\
                    +\infty & \mbox{else}
                \end{cases}
            \end{equation*}
            We shall prove that $(\scr{J}_n)_{n \geq 0}$ $\Gamma$-converges to $\scr{J}[\eta] = \scr{E}_m[\eta] + \frac{1}{2 \tau} W_2^2(\eta,\mu)$ for the narrow convergence in $\cl{P}_2(\Omega)$.
            \begin{enumerate}
                \item \emph{$\Gamma-\liminf$:} If $\eta_n \rightharpoonup \eta$ and $\sup_n \scr{J}_n[\eta_n] < +\infty$, then if $\Omega$ is bounded, we can use l.s.c of the entropy and joint l.s.c of the $W_2$-distance to conclude. On unbounded $\Omega$, we use inequality \ref{eq: lower_bound_JKO} obtained during the proof of Proposition \ref{prop: existence_one_step} to derive uniform upper bound on $(M_2[\rho_n])_n$ (using that $(M_2[\mu_n])_n$ is uniformly bounded by $\bb{W}_2(\Omega)$ convergence), and then conclude using joint l.s.c of $W_2$ and the l.s.c result of Proposition \ref{prop: m_entropy_lsc}. 
                \item \emph{$\Gamma-\limsup$:} Fix $\eta \in \cl{P}_2(\Omega)$ with $\scr{J}[\eta] < +\infty$, and define for $n$ large enough (such that $\eta(\Omega_n) \neq 0$) the measure $\eta_n = \eta(\Omega_n)^{-1} \eta \resmes \Omega_n$, which converges to $\eta$ in $\bb{W}_2(\Omega)$. As $W_2$ is continuous for this convergence, we have $W_2^2(\eta_n,\mu_n) \to W_2^2(\eta,\mu)$, and monotone convergence provides $\scr{E}_m[\eta_n] \to \scr{E}_m[\eta]$. 
            \end{enumerate}
            
            \item \textbf{Narrow convergence of $(\rho_n)_{n \geq 0}$}: We argue that the sequence $(\rho_n)_{n \geq 0}$ is precompact for the narrow convergence of the sequence of minimizers for this convergence. In bounded domain, this is immediate. In unbounded domains, fix some $\eta_0 \in \cl{P}_2(\Omega_0)$ with finite $\scr{E}_m[\eta_0]$, then $\sup_n \scr{J}_n[\rho_n] \leq \scr{E}_m[\eta_0] + \frac{1}{2 \tau}(M_2[\eta_0] + \sup_n M_2[\mu_n]) < +\infty$. Therefore using again inequality \ref{eq: lower_bound_JKO} we deduce that $\sup_n M_2[\rho_n] < +\infty$, providing tightness, and hence precompactness, of the sequence of minimizers. By $\Gamma$-convergence, the sequence $\rho_n$ is converging narrowly, up to subsequence, to $\rho$ minimizer of the One-Step JKO problem starting from $\mu$ (we will now assume that we did this extraction). 
            
            \item \textbf{Upgrading to $\bb{W}_2(\Omega)$-convergence}: Using that the $\Gamma-\liminf$ and $\Gamma-\limsup$ inequalities must be equalities for the sequence $(\rho_n)_{n \geq 0}$, we deduce that we have convergence of the transport distance: $W_2^2(\rho_n,\mu_n) \to W_2^2(\rho,\mu)$. We argue that such convergence is enough to deduce $\bb{W}_2(\Omega)$ convergence of $(\rho_n)_n$ toward $\rho$. To do so, let $\gamma_n$ be an optimal transport plan between $\rho_n$ and $\mu_n$. By standard stability results in optimal transport theory, $\gamma_n$ is converging narrowly to $\gamma$, optimal transport plan between $\rho$ and $\mu$. Next we use that $|x|^2 = 2 |y|^2 + 2|x-y|^2 - |x-2y|^2$ which gives, after integrating against $\gamma_n$,
            \begin{equation*}
                M_2[\rho_n] = 2 M_2[\mu_n] + 2 W_2^2(\rho_n,\mu_n) - \int_{\Omega \times \Omega} |x-2 y|^2 \dd{\gamma_n}
            \end{equation*}
            Now using narrow convergence of $\gamma_n$, we have $\int_{\Omega \times \Omega} |x-2 y|^2 \dd{\gamma} \leq \liminf_n \int_{\Omega \times \Omega} |x-2 y|^2 \dd{\gamma_n}$. Therefore by $\bb{W}_2(\Omega)$-convergence of $(\mu_n)_n$ and convergence of $W_2^2(\rho_n,\mu_n)$ we obtain
            \begin{equation*}
                \limsup_n M_2[\rho_n] \geq 2 M_2[\mu] + 2 W_2^2(\rho,\mu) - \int_{\Omega \times \Omega} |x-2 y|^2 \dd{\gamma} = M_2[\rho]
            \end{equation*}
            which, combined with the l.s.c. of the second moment for the narrow convergence gives $M_2[\rho_n] \to M_2[\rho]$, i.e. $\rho_n \to \rho$ in $\bb{W}_2(\Omega)$. \qedhere
        \end{itemize}
    \end{proof}

    More interestingly, when the domain $\Omega_N$ are bounded, one can say a bit more. Let $u_N := \tau f_m'(\rho_N) + \frac{1}{2} |x|^2$, which, by optimality, is a Brenier's potential from $\rho_N$ to $\mu_N$. We also let $u := \tau f'_m(\rho) + \frac{1}{2} |x|^2$, assuming that $\Omega$ is volume-regular to ensure absolute continuity of minimizers in the case $m < 1$. 

    \begin{proposition}[Convergence of potentials] \label{prop: convergence_potentials}
        We have $u_N \to u$ locally uniformly on $\Omega$. In the sense that if $Q \subset \rm{int}(\Omega)$ is bounded, then $Q \subset \Omega_N$ for any $N$ large enough, and $u_N \to u$ uniformly on $Q$. 
    \end{proposition}

    \begin{proof}
        Let $\bb{B}_2 = B_{2r}(x) \subset \Omega$, contained in all the $\Omega_N$ for $N$ large enough. By optimality condition, $u_N$ is convex, which implies that $\Delta f'_m(\rho_N) \geq -\frac{d}{\tau}$. Using Lemma \ref{lemma: L1-Linfty_regularization} to $\rho_N$ on $Q_{3r}(x)$, we deduce that $(\rho_N)_N$ is uniformly bounded on $\bb{B}_1 = B_r(x)$ (as $\rho_N(Q_r(x)) \leq 1$). 
        \begin{enumerate}
            \item If $m > 1$, then as $f'_m(\rho_N) = \frac{m}{m-1} \rho_N^{m-1}$ is non-negative, we have $u_N \geq \frac{1}{2} |x|^2$, i.e. $u_N$ is uniformly lower-bounded on $\bb{B}_1$. It is also uniformly upper bounded on this set by the uniform upper bound on $\rho_N$. Therefore the sequence $(u_N)_{N \geq 0}$ is uniformly bounded on $\bb{B}_1$. Since the sequence is convex, it converges, up to subsequence, locally uniformly on $\bb{B}_1$, to another convex function $u$. But as $\rho_N = \frac{m-1}{\tau m} \left ( u_N - \frac{1}{2} |x|^2 \right )^\frac{1}{m-1}$, the limiting function must be equal to $u$. 
            
            \item If $m \leq 1$, the upper bound only provides an uniform upper bound on $(u_N)_{N \geq 0}$ on $\bb{B}_1$. We argue that this sequence must also be bounded from below on $\bb{B}_{1/3} = B_{r/3}(x)$. Indeed, suppose that this is not the case, then we can find $x_N \in \bb{B}_{1/3}$ such that $u_N(x_N) \to -\infty$. Let $y \in \bb{B}_{1/3}$, and define $z_N = 2 y - x_N \in B_r(x)$, then we have $y = \frac{1}{2}(z_N + x_N)$, hence
            \begin{equation*}
                u_N(y) \leq \frac{1}{2} u_N(x_N) + \frac{1}{2} u_N(z_N) \leq \frac{1}{2} u_N(x_N) + \frac{1}{2} \sup_{N,\bb{B}_1} u_N 
            \end{equation*}
            Therefore $u_N \to -\infty$ uniformly on $\bb{B}_{1/3}$. Exploiting the relation between $u_N$ and $\rho_N$, this shows that $\rho_N \to 0$ uniformly on $\bb{B}_{1/3}$, hence $\rho = 0$ on this set. But as $\rho > 0$ a.e. this is absurd. Hence the sequence is uniformly lower bounded on $\bb{B}_{1/3}$, and converges, up to subsequence, locally uniformly, to a convex function, which again must be equal to $u$. 
        \end{enumerate}
        Therefore for all $x \in \rm{int}(\Omega)$, $u_N$ converges uniformly to $u$ in a neighborhood of $x$, which concludes. 
    \end{proof}

\section{One-Step Improvement of Monge-Amp\`ere lower bound} \label{section: one_step_improvement}

    The first step toward our result is to show that if $\mu$ is regular enough and satisfies a Monge-Amp\`ere lower bound, then after one step of the JKO scheme, this lower bound will be improved. 

\begin{theorem}[One-Step Improvement of Monge-Amp\`ere lower bound] \label{thm: one_step_improvement}
    Suppose that $\Omega$ is a cube or the torus in dimension $d \in \{ 1, 2 \}$. Let $\mu \in \cl{P}_{ac}(\Omega) \cap C^2(\Omega)$, strictly positive, and such that $u_m[\mu] = \frac{1}{2} |x|^2 + \tau f_m'(\mu)$ is convex. Then the same holds for $\rho = Q_m^\tau[\rho]$. 
    
    Furthermore, suppose that $\det(D^2 u_m[\mu])^\frac{1}{d} \geq \lambda \geq 0$, and additionally in the case of a cube, that $\nabla \mu(x) \cdot n = 0$ for all $x \in \partial Q \setminus \cl{C}$, and $\nabla \eta(x) = 0$ for all $x \in \cl{C}$ (see Section \ref{subsection: treating_boundary} for the definition of $\cl{C}$). Then one of the following holds:
    \begin{enumerate}
        \item $\det(D^2 u_m[\rho])^\frac{1}{d} \geq 1$.
        \item There exists $\Lambda > 0$ such that $\det(D^2 u_m[\rho])^\frac{1}{d} \geq \Lambda$ and
        \begin{equation}
            1 + \frac{1}{\Lambda^{d(m-1) + 1}} - \frac{1}{\Lambda^{d m}} \geq \lambda
        \end{equation}
    \end{enumerate}
\end{theorem}

\begin{proof}
    As $\mu$ is $C^2(\Omega)$ and strictly positive, there exists $\epsilon > 0$ such that $\epsilon \leq \mu \leq \epsilon^{-1}$. Under this regularity assumption, if $\rho = Q_m^\tau[\mu]$, we have the following:
    \begin{enumerate}
        \item We have $\epsilon \leq \rho \leq \epsilon^{-1}$, and if $(\psi,\phi)$ is a pair of Kantorovich potentials from $\rho$ to $\mu$, then $\tau f_m'(\rho) = -\psi$. In particular, $u_m[\rho]$ is a Brenier potential from $\rho$ to $\mu$, and therefore convex.
        \item $\rho,\psi$ are of class $C^{3,\beta}(\Omega)$ for some $\beta \in (0,1)$ and $u_m[\rho]$ is strictly convex, in particular, $\det(D^2 u_m[\rho]) > 0$ on $\Omega$. 
    \end{enumerate}
    The first point is a consequence of \ref{prop: propagation_upper_lower_bound}. The second one follows from a bootstrap argument: the bound away from $0$ and $+\infty$ for $\rho,\mu$ provides $\psi \in C^{1,\beta}(\Omega)$ for some $\beta < 1$, and the corresponding Brenier potential is strictly convex, which in turn gives $\rho \in C^{1,\beta}(\Omega)$ by the optimality condition and the uniform bounds on $\rho$. Combining this with $\mu \in C^2(\Omega) \subset C^{1,\beta}(\Omega)$, we deduce that $\psi$, and hence $\rho$, are of class $C^{3,\beta}$. 
    
    In particular, under these regularity assumptions, using that $u_m[\rho]$ is a Brenier potential, the following Monge-Amp\`ere equation holds in the classical sense:
    \begin{equation*}
        \det(D^2 u_m[\rho]) = \frac{\rho}{\mu(\nabla u_m[\rho])}
    \end{equation*}
    We shall argue at a minimum point for $J := \det(D^2 u_m[\rho])$, and let $\Lambda^d$ be the minimum of $J$, which is strictly positive by strict convexity of $u_m[\rho]$. We will first treat the case of an interior minimum, and then treat the case where the minimum is attained at a boundary point.
    
    For simplicity, we shall write $u$ for $u_m[\rho]$ and $v$ for $u_m[\rho]$. We also let $p := \tau f_m'(\rho)$ and $q := \tau f_m'(\mu)$, so that $u = \frac{1}{2} |x|^2 + p$, $v = \frac{1}{2} |x|^2 + q$. Furthermore, applying $f'_m$ to the Monge-Amp\`ere equation gives 
    \begin{equation} \label{eq: Monge_Ampere_J}
        \begin{cases}
            J^{m-1} = \det(D^2 u)^{m-1} = \frac{p}{q(\nabla u)} & m \neq 1 \\
            \log J = \log \det(D^2 u) = p - q(\nabla u) & m = 1
        \end{cases}
    \end{equation}
    
    Let $x_0$ be a minimizer of $J$, and suppose that $x_0$ is in the interior of $\Omega$. From now on, all computations shall be performed at this particular point $x_0$. 

    \begin{itemize}
        \item \textbf{Step 1: Second-order optimality conditions at $x_0$}: Taking the Hessian of the Monge-Amp\`ere equation, and letting $R := D^2 u\, D^2 q(\nabla u)\, D^2 u + D^2\nabla u \cdot \nabla q(\nabla u)$, we get
        \begin{equation*}
            D^2 p = \begin{cases}
                J^{m-1} R 
                + (m-1)J^{m-2} \nabla J \cdot D^2 u\,\nabla q(\nabla u) 
                + q(\nabla u)\, D^2(J^{m-1}) & m \neq 1 \\[6pt]
                R + D^2(\log J) & m = 1
            \end{cases}
        \end{equation*}
        where $D^2 \nabla u \cdot \nabla q(\nabla u) = \sum_i D^2 (\partial_i u) \cdot \partial_i q(\nabla u)$. 

        For a symmetric matrix $A$, we shall write $A \succeq 0$ if $A$ is positive semi-definite. We argue that for each choice of $m$, the last term in each respective case is positive semi-definite. Indeed:
        \begin{itemize}
            \item For $m > 1$, $x_0$ is also a minimum of $J^{m-1}$, hence $D^2 J^{m-1} \succeq 0$, and since $q(\nabla u) \geq 0$ (by $q = \frac{\tau m}{m-1} \mu^{m-1} \geq 0$), we have $q(\nabla u) D^2 J^{m-1} \succeq 0$.
            \item For $m = 1$, $x_0$ is also a minimum of $\log J$, hence $D^2 \log J \succeq 0$.
            \item For $m < 1$, $x_0$ is now a maximum of $J^{m-1}$, hence $D^2 J^{m-1} \preceq 0$. Since $q = \frac{\tau m}{m-1} \mu^{m-1}$ and $m-1 < 0$, we have $q \leq 0$, therefore $q(\nabla u) D^2 J^{m-1} \succeq 0$. 
        \end{itemize}
        Using the first-order optimality condition $\nabla J = 0$, we get the inequalities
        \begin{equation*}
            \begin{cases}
                D^2 p \succeq J^{m-1}[D^2 u D^2 q(\nabla u) D^2 u + D^2 \nabla u \cdot \nabla q(\nabla u)] & m \neq 1 \\
                D^2 p \succeq D^2 u D^2 q(\nabla u) D^2 u + D^2 \nabla u \cdot \nabla q(\nabla u) & m = 1
            \end{cases}
        \end{equation*}
        To eliminate the third-order term, we shall use that, by first-order optimality, $0 = J^{-1} \nabla J = \nabla \log J = \Tr [D^2 u]^{-1} D^2 \nabla u$. Since $D^2 u \succeq 0$, taking the trace against $[D^2 u]^{-1}$ in the previous inequalities gives
        \begin{equation} \label{eq: ineq_optimal_point}
            \Tr D^2 p [D^2 u]^{-1} \geq J^{m-1} \Tr D^2 q(\nabla u) D^2 u 
        \end{equation}
        We shall next exploit this algebraic matrix inequality.

        \item \textbf{Step 2: Exploiting inequality \ref{eq: ineq_optimal_point}}: Replacing $D^2 p$ by $D^2 u - \rm{I}_d$, $D^2 q$ by $D^2 v - \rm{I}_d$, and $J$ by $\Lambda^d$, we obtain
        \begin{equation*}
            1 - \frac{1}{d} \Tr [D^2 u]^{-1} + \frac{\Lambda^{d(m-1)}}{d} \Tr D^2 u \geq \frac{\Lambda^{d(m-1)}}{d} \Tr D^2 v(\nabla u) D^2 u
        \end{equation*}
        Next, we use the classical AM-GM inequality for matrices, which gives $\frac{1}{d} \Tr AB \geq \det(A)^\frac{1}{d} \det(B)^\frac{1}{d}$ for positive symmetric matrices $A,B$. This gives $\frac{1}{d} \Tr D^2 v(\nabla u) D^2 u \geq \det(D^2 v(\nabla u))^\frac{1}{d} \det(D^2 u)^\frac{1}{d}$, which can be bounded from below by $\lambda \cdot \Lambda$, using $\det(D^2 v(\nabla u))^{1/d} \geq \lambda$ by assumption and the definition of $\Lambda$. Therefore, we get
        \begin{equation} \label{eq: ineq_optimal_point_2}
            1 - \frac{1}{d} \Tr [D^2 u]^{-1} + \frac{\Lambda^{d(m-1)}}{d} \Tr D^2 u \geq \Lambda^{d(m-1)+1} \cdot \lambda
        \end{equation}
        We now distinguish between the case $d=1$ and $d=2$.

        \item \textbf{Step 3: Dimension one}: In dimension one, we have $D^2 u = u''$ and $\det(D^2 u) = u'' = \Lambda$. Hence we obtain
        \begin{equation*}
            1 - \frac{1}{\Lambda} + \Lambda^{d(m-1) + 1} \geq \Lambda^{d(m-1) + 1} \cdot \lambda
        \end{equation*}
        and the result follows by dividing by $\Lambda^{d(m-1) + 1}$.

        \item \textbf{Step 4: Dimension two}: In dimension two, we exploit the equality $\Tr A^{-1} = \frac{\Tr A}{\det A}$, valid for any invertible symmetric matrix $A$ (and specific to this dimension), which follows by diagonalization. We obtain
        \begin{equation*}
            1 + \frac{1}{d} \Tr D^2 u \left ( \Lambda^{d(m-1)} - \frac{1}{\Lambda^d} \right ) \geq \Lambda^{d(m-1)+1} \cdot \lambda
        \end{equation*}
        If $\Lambda \geq 1$, there is nothing to do, otherwise we have $\Lambda^{d(m-1)} - \frac{1}{\Lambda^d} \leq 0$, and we can again use the AM-GM inequality, which gives $\frac{1}{d} \Tr D^2 u \geq \det(D^2 u)^\frac{1}{d} = \Lambda$ to get
        \begin{equation*}
            1 + \Lambda^{d(m-1) + 1} - \frac{1}{\Lambda^{d-1}} \geq \Lambda^{d(m-1)+1} \cdot \lambda
        \end{equation*}
        and we conclude again by dividing by $\Lambda^{d(m-1) + 1}$. 
    \end{itemize}
    This conclude the proof in case of interior minimum point. 
\end{proof}

    \begin{remark}
        The reason for restricting to small dimension is apparent in the proof: one needs to control the determinant from below using both the trace of the Hessian and its inverse. This is possible only for dimension at most $2$, and in larger dimension, it is possible to construct examples of matrices $A,B \succ 0$ with $\det(A) \geq \lambda$ and $\det(B)$ arbitrarily small, such that, upon replacing all $D^2 v(\nabla u)$ by $A$, and $D^2 u$ by $B$, the algebraic inequality \ref{eq: ineq_optimal_point_2} holds true. The typical case is, in dimension $3$, to take $B = \rm{Diag}(\epsilon,\epsilon,\epsilon^r)$ for some well chosen $r$, and $A = c \rm{I}_d$ for some constant $c > 0$. 

        On the other hand, we expect $D^2 u$ to be close to identity (as we expect $D^2 f_m'(\rho)$ to be of order $1$), and then one should in principle be able to linearize the inequality. It is, however, unclear that such a linearization can be done as the order-one estimate on $D^2 u_m$ is only formal, or strongly depends on $\mu$, which would limit its usefulness for treating the case of general initial data when iterating the estimate. 
    \end{remark}

    This analysis assumed that $x_0$ interior point; on the cube, however, the minimum can be attained at the boundary. For PDEs, this is typically handled using the Hopf lemma. Here, however, we shall need a more involved analysis. On the other hand, a careful inspection of the proof shows that it suffices to prove that even if $x_0$ is at the boundary, we have $\nabla J(x_0) = 0$, and $\Tr [D^2 u(x_0)]^{-1} D^2 J(x_0) \geq 0$, thereby relaxing the full $D^2 J(x_0) \succeq 0$ hypothesis. 
    
\subsection{Treating the boundary} \label{subsection: treating_boundary}
        The previous analysis was conducted assuming the minimum was attained in the interior of the domain. In order to tackle the general case, we shall take care of a boundary minimimum point. 
        
        We let $\Omega = Q$ be a cube in dimension one or two. We denote by $\cl{C}$ the set of corners of $Q$, and we call face of $Q$ the closures of any connected components of $\partial Q \setminus \cl{C}$ (in dimension $2$). We say that a point is in the interior of a face if it belongs to a face, and is not a corner. If $x$ is such a point, we let $F_x$ be the unique face such that $x \in F_x$. We also consider a convex function $h : \bb{R}^d \to \bb{R}$ such that $Q = \{ h \leq 0 \}$ with equality at the boundary, $h$ of class $C^1(\bb{R}^d \setminus \cl{C})$ and $\nabla h(x) = n(x)$ the outward pointing normal at any point $x$ in the interior of a face. We shall note that $n$ is in fact constant on any face $F$, and $n \perp (F-F)$. 
        
        We let $N_x Q$ be the tangent cone of $Q$ at a point $x \in Q$, which is defined as the set of all $v \in \bb{R}^d$ such that for some sequence $y_n \in Q$ and $t_n > 0$ converging to $0$ one has $y_n = x + t_n v + o(t_n)$. As $Q$ is convex, this coincides with the set of admissible directions, i.e. $v \in \bb{R}^d$ such that $x + t v \in \Omega$ for all $t$ small enough. Using Taylor's expansion, we easily see that if $\Phi$ is a $C^1$-diffeomorphism of $\Omega$, then $D \Phi(x) N_x \Omega = N_{\Phi(x)} \Omega$. 

        We recall the following optimality condition result at boundary points:
        
        \begin{lemma}[Boundary optimality condition] \label{lemma: boundary_optimality}
            Let and suppose that $x \in \partial Q$ is a minimum (resp. maximum) point of a function $f$. Then if $f$ is $C^1$ near $x$:
            \begin{enumerate}
                \item For all $v \in N_x Q$, we have $\nabla f(x) \cdot v \geq 0$ (resp. $\leq 0$). (in other word, $\nabla f(x)$ is in the polar cone of $N_x Q$). 
                \item If $f$ is $C^2$, then for all $v \in N_x Q$ is such that $\nabla f(x) \cdot v = 0$. Then $D^2 f(x)[v,v] \geq 0$ (resp. $\leq 0$).
                \item If $x \in \partial Q \setminus \cl{C}$, there exists $\lambda \leq 0$ (resp. $\geq 0$) such that $\nabla f(x) = \lambda \cdot n(x)$. 
            \end{enumerate}
        \end{lemma}

        \begin{proof}
            We shall only consider the minimum case. The first two points follows by a Taylor expansion up to first and second order, as $0 \leq f(x+tv) - f(x) = t \nabla f(x) \cdot v + t^2 \frac{1}{2} D^2 f(x)[v,v] + o(t^2)$ for all $t > 0$ small enough. For the last point, we use $N_x Q = \{ v \in \bb{R}^d, v \cdot n(x) \leq 0 \}$
        \end{proof}

        The reason of using a cube instead of a general domain lies in the following result on the behavior of any optimal transport map on the boundary of such a set. It would be interesting to see if such a result also holds on polygonal domains, using for example the recent regularity theory for general convex domain developed in \cite{RegCaffConvexDomain}, but this would need the adaptation of the tangent cone argument to a non-$C^1$-setting. 

        \begin{proposition}[Behavior of transport on $\partial Q$] \label{prop: behavior_map_boundary}
            Let $\nu,\mu$ be two probability measures with strictly positive densities of class $C^{0,\alpha}(Q)$ for some $\alpha < 1$, and let $T$ be the optimal transport map from $\nu$ to $\mu$. Then
            \begin{enumerate}
                \item $T$ is the identity when restricted to the set of corners. 
                \item In dimension $2$, if $F$ is a face of $Q$, then $T(F) = F$. 
            \end{enumerate}
        \end{proposition}

        \begin{proof}
            By Caffarelli's regularity on the cube \ref{thm: Caff_regu}, $T$ is a $C^{1,\alpha}$-diffeomorphism of $Q$, therefore $D T(x)(N_x Q) = N_{T(x)} Q$ for any $x \in Q$. We argue that this implies that corner are sent to corner, interior face point to interior face point, and interior point to interior point. A simple way to see this is to look at the largest dimension of a subspace contained $N_x Q$, which completely characterize corners, interior faces, and interior point, which is an algebraic invariant by linear invertible map (such as $D T(x)$). 

            Now, for the first point, we notice that, by bijectivity of $T$, $T(\cl{C}) = \cl{C}$. By \cite[Theorem 1.38]{OT_Filippo} the support of the optimal transport plan from $\nu$ to $\mu$, equal to $\{ (x, T(x)), x \in Q \}$, is $c$-cyclically monotone. By restriction, the same holds for $\{ (x, T(x)), x \in \cl{C} \}$. But then, using \cite[Theorem 1.49]{OT_Filippo}, $T$ is the optimal transport map from $\sum_{c \in \cl{C}} \delta_c$ to itself, i.e. it must be the identity on $\cl{C}$. 

            Let's finally prove the last point. Let $F$ be a face, and write it down $F = [c_1,c_2]$ for two corners $c_1,c_2 \in \cl{C}$. Then we already now that $T(]c_1,c_2[) \subset \partial Q \setminus \cl{C}$. As $T$ is continuous, $T(]c_1,c_2[)$ is connected, therefore it is fully contained in a face $G$. But then $T([c_1,c_2])$ is also contained in $G$, and contains $[T(c_1),T(c_2)] = [c_1,c_2] = F$ by connectedness. Hence $F \subset T(F) \subset G$, which forces $F = G$ and concludes the proof. 
        \end{proof} 

        We now have all the ingredients needed to finish the proof. 

        \begin{proof}[Ending proof of Theorem \ref{thm: one_step_improvement}]
            As above, we let $\Lambda := \min_{x \in Q} \det(D^2 u)^{1/d}$, and $x_0$ some minimizer. The goal is to prove, if $x_0$ is not an interior point, then at $x_0$ the inequality \ref{eq: ineq_optimal_point} still holds true. We let $T = \nabla u$ be the optimal transport map from $\rho$ to $\mu$, which satisfies the hypothesis of Proposition \ref{prop: behavior_map_boundary}. We divide the reasoning into several steps.
            \begin{itemize}
                \item \textbf{Step 1: Gradient at corners}: First notice that $x + \tau \nabla p = \nabla u = T(x)$ for all $x \in Q$. Since $T(x) = x$ at corners, we must have $\nabla p = 0$ on the corners. Taking the gradient of equation \ref{eq: Monge_Ampere_J} we get
                \begin{equation*}
                    \begin{cases}
                        (m-1) J^{m-2} \nabla J = \frac{\nabla p}{q(\nabla u)} - J D^2 u \nabla q(\nabla u) & m \neq 1 \\
                        J^{-1} \nabla J = \nabla p - D^2 u \nabla q(\nabla u) & m = 1
                    \end{cases}
                \end{equation*}
                Applying this at corner point we get, for $c_m = m-1$ for $m \neq 1$ and $1$ else $c_m J^{m-3} \nabla J = -D^2 u \nabla q$. Since, by assumption, we have $\nabla q = 0$ at corners, we deduce that $\nabla J(x) = 0$. 

                \item \textbf{Step 2: Gradient at interior face points}: The function $h(\nabla u)$ is maximized at points of $\partial Q$. Furthermore, since $\nabla u(x) \in \partial Q \setminus \cl{C}$ for any $x \in \partial Q \setminus \cl{C}$, $h(\nabla u)$ is $C^1$ near any point in the interior of a face. By Lemma \ref{lemma: boundary_optimality}, there must be a function $x \to \lambda(x) \geq 0$ such that 
                \begin{equation*}
                    \nabla [h(\nabla u)] = D^2 u \nabla \cdot h(\nabla u) = D^2 u \cdot n = \lambda(x) \cdot n(x)
                \end{equation*}
                at any $x \in \partial Q \setminus \cl{C}$. Furthermore, since $x \cdot n = 0$ for any such $x$, we also have $T(x) \cdot n = 0$. Since $\nabla p(x) = T(x) - x$ we deduce that $\nabla p(x) \cdot n = 0$. Putting this into the Monge-Amp\`ere equation we obtain on $\partial Q \setminus \cl{C}$
                \begin{align*}
                    c_m J^{m-3} \nabla J(x) \cdot n(x) &= -D^2 u \nabla q(\nabla u(x)) \cdot n(x) = -\nabla q(\nabla u(x)) \cdot (D^2 u(x) \cdot n(x)) \\
                    &= -\lambda(x) \nabla q(\nabla u(x)) \cdot n(x) = \lambda(x) \nabla q(\nabla u(x)) \cdot n(\nabla u(x)) = 0
                \end{align*}
                by assumption, were we use that $\nabla u(x)$ is on the same face as $x$. Hence we obtain $\nabla J \cdot n = 0$ on $\partial Q \setminus \cl{C}$. 

                \item \textbf{Step 3: The case of an interior face minimum point}: Consider $d=2$ and assume that $x_0 \in \partial Q \setminus \cl{C}$. By Lemma \ref{lemma: boundary_optimality}, there exists $\lambda \leq 0$ such that $\nabla J(x_0) = \lambda n(x_0)$. But by the previous discussion, we have $\nabla J(x_0) \cdot n(x_0) = 0$. Hence $\lambda = 0$ and we deduce that $\nabla J(x_0) = 0$. Using again Lemma \ref{lemma: boundary_optimality}, we obtain $D^2 J(x_0)[v,v] \geq 0$ for all $v \in N_{x_0} Q$. This is then also true for all $v \in -N_{x_0} Q$ (as $D^2 J(x_0)[-v,-v] = D^2 J(x_0)[v,v]$). As $N_{x_0} Q \cup (-N_{x_0} Q) = \bb{R}^d$, we deduce that $D^2 J(x_0) \succeq 0$ and $\nabla J(x_0) = 0$. Therefore we can deduce the inequality \ref{eq: ineq_optimal_point} as in the interior point case and we conclude. 

                \item \textbf{Step 4: The case of a corner minimum point}: We notice that in order to derive the inequality \ref{eq: ineq_optimal_point} we only need $\nabla J(x_0) = 0$ and $\Tr [D^2 u(x_0)]^{-1} D^2 J \geq 0$. The first point being proved in Step 1, we shall prove the second. By Lemma \ref{lemma: boundary_optimality}, for all $v \in N_{x_0} Q$ we get $D^2 J(x_0)[v,v] \geq 0$. If $d=1$, this implies that $D^2 J = J'' \geq 0$ and the reasoning is done. If $d=2$, we can, up to rotation, assume that $x = (0,0)$. Notice that $\nabla u(t,0) = (s(t),0)$ for some $s(t)$. In particular, $\partial_2 u(t,0) = 0$, hence $\partial_{12} u(t,0) = 0$. Letting $t \to 0$ gives $\partial_{12} u(0,0) = 0$. Therefore $D^2 u$ is diagonal in the canonical basis. But since $e_1,e_2 \in N_x Q$ we have 
                \begin{equation*}
                    \Tr [D^2 u(x_0)]^{-1} D^2 J(x_0) = \frac{J_{11}(x_0)}{u_{11}(x_0)} + \frac{J_{22}(x_0)}{u_{22}(x_0)} \geq 0
                \end{equation*}
                which concludes the proof. \qedhere
            \end{itemize}
        \end{proof}

\section{Proof of the main Theorem} \label{section: iteration}

    We let $\Omega$ being either the torus, a cube, a quarter-space a half-space or the whole space in dimension $1$ or $2$ (which are all volume-regular). We assume that $m > m_c^1 = 1 - \frac{2}{d}$ if $\Omega$ is bounded, and $m > m_c^2 = $ if $\Omega$ is unbounded. 

We introduce the function
\begin{equation*}
    F_{d,m}[X] := \frac{1}{(1-X)^{d(m-1)+2}} - \frac{1}{(1-X)^{d(m-1)+1}} = \frac{X}{(1-X)^{d(m-1)+2}}
\end{equation*}
We shall note that this function has the following properties:
\begin{itemize}
    \item $F_{d,m}$ is increasing, and define a bijection from $[0,1)$ to $[0,+\infty)$.
    \item $F_{d,m}[X] \geq X$ for all $X \in [0,1)$. 
\end{itemize}
In particular, we can define a sequence $(X_k)_{k \geq 1}$ of $[0,1]$ by $X_1 = 1$, and $X_{k+1} = F_{d,m}^{-1}[X_k]$. 

We can now prove the main theorem of the paper: the asymptotic Aronson-Bénilan estimate in the JKO scheme.

\begin{theorem}[Asymptotic Aronson-Bénilan estimate] \label{thm: main_result}
    Let $\rho_0 \in \cl{P}_2(\Omega)$, consider the iteration of the JKO scheme starting from $\rho_0$: $\rho_{k+1}^\tau = Q_m^\tau[\rho_k^\tau]$. We let $(\rho_t^\tau)_{t \geq 0}$ be the piecewise constant interpolation of the values of the $(\rho_k^\tau)_{k \geq 0}$, i.e. $\rho_t^\tau = \rho_k^\tau$ on $[k \tau, (k+1) \tau)$.
    \begin{enumerate}
        \item The function $u_k^\tau := \tau f_m'(\rho_k^\tau) + \frac{1}{2} |x|^2$ is convex for all $k \geq 1$. 
        \item For all $k \geq 1$ we have the inequality, in the Monge-Amp\`ere sense, for the sequence $(X_k)_{k \geq 1}$ defined above:
        \begin{equation}
            \det(D^2 u_k)^\frac{1}{d} \geq 1 - X_k
        \end{equation}
        and we have $X_k \sim \frac{1}{(d(m-1)+2) k}$ as $k \to +\infty$.
        \item The following asymptotic Aronson-Bénilan estimate holds: for all $t_0 > 0$, and $\epsilon > 0$ there exists $\delta$ such that for all $\tau < \delta$, $t \geq t_0$ we have:
        \begin{equation}
            \Delta f'_m(\rho_t^\tau) \geq -(1+\epsilon) \frac{d}{d(m-1) + 2} \cdot \frac{1}{t}
        \end{equation}
    \end{enumerate}
\end{theorem}

\begin{proof}

    We divide into several steps.

    \begin{itemize}
    
        \item \textbf{Bounded domain, regular initial data}: We first prove points $(1)$ and $(2)$ under strong regularity assumption on the initial data, on the torus and on the cube. More precisely, we assume that $\rho_0$ is of class $C^2(\Omega)$ and positive. We check that $\rho_k^\tau$ satisfies the hypothesis of Theorem \ref{thm: one_step_improvement}:
        \begin{enumerate}
            \item By propagation of lower and upper bound, there exists some $\epsilon > 0$ such that $\epsilon \leq \rho_k^\tau \leq \epsilon^{-1}$ for all $k \geq 0$.
            \item By the optimality conditions, $u_k^\tau = \tau f_m'(\rho_k^\tau) + \frac{1}{2} |x|^2$ is convex for all $k \geq 1$. 
            \item Therefore $f_m'(\rho_k^\tau)$ is Lipschitz, which combined with the upper and lower bounds, implies that $\rho_k^\tau$ is itself Lipschitz.
            \item If $\rho_k^\tau$ is of class $C^2(\Omega)$, then as $\rho_{k+1}^\tau$ is Lipschitz, Caffarelli's regularity implies that the Kantorovitch potentials are $C^2(\Omega)$, which in turn implies that $\rho_{k+1}^\tau$ is $C^2(\Omega)$ by optimality conditions. Since $\rho_0$ is $C^2(\Omega)$, we can propagate this regularity: $\rho_k^\tau$ is $C^2(\Omega)$ for all $k \geq 0$. 
            \item Finally, if $\Omega$ is the cube, then by optimality conditions, we have for $k \geq 1$, $\tau \nabla f_m'(\rho_k^\tau) = T_k - \rm{id}$ where $T_k$ is the optimal transport map from $\rho_k^\tau$ to $\rho_{k-1}^\tau$. Then if $x$ is a corner, we obtain $\nabla f_m'(\rho_k^\tau)(x) = 0$ since $T_k(x) = x$, and if $x$ is in the interior of a face $F$, then $T_k(x)$ is in the interior of the same face, and we have $\nabla f_m'(\rho_k^\tau)(x) \cdot n(x) = T_k(x) \cdot n_F - x \cdot n_F = 0$ as $n \perp (F-F)$. 
        \end{enumerate}
        
        Therefore we can iterate the One-Step improvement of Monge-Amp\`ere lower bound starting from $\rho_1^\tau$.
        
        Let $\Lambda_k$ is the infimum of $\det(D^2 u_k^\tau)^{1/d}$, then either $\Lambda_{k+1} \geq 1$, or $F_{d,m}[1-\Lambda_{k+1}] \leq 1 - \Lambda_k$. We argue that $\Lambda_k \geq 1 - X_k$ for all $k \geq 0$. Indeed, this is true for $k = 1$ as $\Lambda_1 \geq 0$ by convexity of $u_1^\tau$. Suppose that this is true for some $k \geq 1$, then using the One-Step improvement \ref{thm: one_step_improvement}, we either have $\Lambda_{k+1} \geq 1$, which gives $\Lambda_{k+1} \geq 1 - X_{k+1}$ as $X_{k+1} \geq 0$, or, by algebraic manipulation of the inequality appearing in the Theorem, $F_{d,m}[1-\Lambda_{k+1}] \leq 1-\Lambda_k \leq X_k = F_{d,m}[X_{k+1}]$, which yields $1-\Lambda_{k+1} \leq X_{k+1}$ as $F_{d,m}$ is increasing, and we conclude by induction. 
        
        \item \textbf{General case} We now approximate the domain, if unbounded, by an increasing sequence of cubes $(\Omega_N)_{N \geq 0}$, and the initial datum by $C^2(\Omega_N)$ positive initial datum. Iterating Proposition \ref{prop: stability_JKO}, the iterates are converging in $\bb{W}_2(\Omega_N)$, and by Proposition \ref{prop: convergence_potentials}, the convex functions $(u_{k,N}^\tau)_{N \geq 0}$ are converging locally uniformly to $u_k^\tau$. Using the stability of the Monge-Amp\`ere measure, this concludes the general case. 

        \item \textbf{Asymptotic for $(X_k)_{k \geq 0}$} By $F_{d,m}[X] \geq X$, we deduce that the sequence $(X_k)_{k \geq 0}$ is decreasing, hence converging. And one easily see that the only fixed point of $F_{d,m}$ is $0$. To obtain the asymptotic, we observe that one can linearize $F_{d,m}$ around $0$ as $F_{d,m}[X] = X(1 + \alpha X + o(X))$, for $\alpha = d(m-1) + 2$. Therefore we have
        \begin{equation*}
            X_k = X_{k+1}(1 + \alpha X_{k+1} + o(X_{k+1}))
        \end{equation*}
        as $k \to +\infty$. Hence
        \begin{align*}
            \frac{1}{X_k} = \frac{1}{X_{k+1}} \cdot \frac{1}{1 + \alpha X_{k+1} + o(X_{k+1})} = \frac{1}{X_{k+1}} - \alpha + o(1)
        \end{align*}
        we deduce that $\frac{1}{X_{k+1}} - \frac{1}{X_k} \to \alpha$ as $k \to +\infty$. Applying Ces\`aro lemma, we deduce that $\frac{1}{k X_{k+1}} \to \alpha$, i.e. $X_{k+1} \sim \frac{1}{\alpha k}$.

        \item \textbf{Asymptotic Aronson-Bénilan} Using the AM-GM inequality for Monge-Amp\`ere measure \ref{lemma: AM_GM_Monge_Ampere}. We have $\Delta f'_m(\rho_k^\tau) \geq -\frac{d X_k}{\tau}$. Fix $\epsilon > 0$, for all $k \geq k_0$ large enough, we have $X_k \leq \frac{1+\epsilon}{\alpha (k+1)}$, hence if $t_0 \geq (k_0 + 1) \tau$, i.e. $\tau \leq \frac{k_0+1}{t_0}$, then for $t \geq t_0$ we deduce that 
         \begin{equation*}
             \Delta f'_m(\rho_t^\tau) \geq -d \frac{1+\epsilon}{\alpha \tau (k+1)} \geq -(1+\epsilon) \frac{d}{d(m-1) + 2} \cdot \frac{1}{t} 
         \end{equation*} \qedhere
    \end{itemize} 
\end{proof}

\bibliography{reference}{}
\bibliographystyle{plain}

\end{document}